\documentclass[a4paper,11pt]{amsart}
\usepackage[latin2,utf8]{inputenc}
  \usepackage{xcolor}

\usepackage{amsmath,amssymb,amsthm, amsfonts}
\usepackage{verbatim}
\usepackage{graphicx}
\usepackage{arydshln}
\usepackage{xcolor}

 \numberwithin{equation}{section}
\numberwithin{figure}{section}
\numberwithin{table}{section}

\newcommand{\vr}{\varrho}
\newcommand{\R}{\mathbb{R}}

\newcommand{\cC}{\mathcal{C}}
\renewcommand{\div}{{\rm div}\,}
\newcommand{\dx}{\,{\rm d}x}
\newcommand{\dy}{\,{\rm d}y}
\newcommand{\ds}{\,{\rm d}s}
\newcommand{\dt}{\,{\rm d}t}
\newcommand{\dtau}{\,{\rm d}\tau}
\newcommand{\de}{\partial}
\newcommand{\dvr}{\delta\! \vr}
\newcommand{\dv}{\delta\! v}
\newcommand{\du}{\delta\! u}
\newcommand{\dX}{\delta\! X}
\newcommand{\dA}{\delta\! A}

\newcommand{\dQ}{\delta\! Q}
\newcommand{\ep}{\varepsilon}

\def\cC{{\mathcal C}}

\def\divx{\, \hbox{\rm div}_x\,  }
\def\divy{\, \hbox{\rm div}_y\,  }
\def\divu{\, \hbox{\rm div}_u\,  }
\newcommand{\supp}{{\rm supp}\;}

\def\Id{\hbox{\rm Id}}

\renewcommand{\div}{\mbox{\rm div}\;\!}

\newcommand{\with}{\quad\hbox{with}\quad}
\newcommand{\andf}{\quad\hbox{and}\quad}

\newtheorem{thm}{Theorem}
\newtheorem{lem}{Lemma}
\newtheorem{cor}{Corollary}

\begin{document}

\title[
Stability for inhomogeneous Navier-Stokes equations]{
Stability of the density patches problem with vacuum for incompressible inhomogeneous viscous flows}
\author[R Danchin, PB Mucha and T Piasecki]{R. Danchin$^*$, P.B. Mucha$^\dag$ and T. Piasecki$^\dag$}

\thanks{$^*$ Laboratoire d’Analyse et de Mathématiques Appliquées, Universit\'e Paris-Est Cr\'eteil}
\thanks{$^\dag$Institute of Applied Mathematics and Mechanics, University of Warsaw}

\begin{abstract} 
We consider the inhomogeneous incompressible Navier-Stokes system in a smooth two or three dimensional bounded domain, in the case where the initial density is only bounded. 
Existence and uniqueness for such initial data was shown recently in \cite{DM1}, but the stability issue was left open. 
After observing that the solutions constructed in \cite{DM1} 
have  exponential decay, a result of independent interest, 
we prove the stability with respect to initial 
data, first in Lagrangian coordinates, and then in the  Eulerian frame. We actually  obtain  stability in $L_2(\R_+;H^1(\Omega))$ for the velocity and in a negative Sobolev space for the density.  

Let us underline that, as opposed to prior works, 
in case of vacuum, our stability estimates are not weighted by the initial densities. Hence, our result applies in particular to the classical density patches problem, where the density is a characteristic function.
\end{abstract}

\maketitle
\smallbreak\noindent {\bf AMS subject classification: 35Q30, 76N10}

\section{Introduction} \label{sec:intro}


\bigbreak

We are interested in the following {\em inhomogeneous incompressible} Navier-Stokes system: 
\begin{equation} \label{sys} 
\left\{ \begin{array}{lr} 
\vr_t+v\cdot \nabla\vr=0 & {\rm in } \ \:\R_+\times\Omega,  \\[1ex]
\vr v_t+\vr v \cdot\nabla v-\Delta v +\nabla P=0 & {\rm in }\  \:\R_+\times \Omega,\\[1ex]
\div v=0 & {\rm in }\ \: \R_+\times \Omega.
\end{array}
\right.
\end{equation}
This system describes the motion of incompressible fluids with variable density and originates from simplified models in geophysics. 
The unknowns in the above system are the velocity $v$, the density $\vr$ and the pressure $P,$ depending on the time variable $t\geq0$ 
and on the space variable $x\in\Omega$ where the fluid domain $\Omega$ 
 is  a smooth bounded  subset  of $\R^d$ with  the physical  dimensions $d=2,3.$ 
 
The system is supplemented with the initial data 
\begin{equation} \label{ic}
    \vr|_{t=0}=\vr_0 \andf  v|_{t=0} =v_0.
\end{equation}
At the boundary, we prescribe the no-slip condition
\begin{equation} \label{bc}
v|_{\de\Omega} = 0.    
\end{equation} 

The existence of weak solutions to \eqref{sys} is nowadays well understood and the state of the art on this issue is rather similar to that of the classical incompressible Navier-Stokes system (i.e. with constant density). The analysis goes back to the work of A. Kazhikhov \cite{Kaz} who showed global existence of weak solutions for initial density bounded away from zero. This constraint was removed by J. Simon in \cite{Simon}. Later, P.-L. Lions \cite{PLL} showed that the density is a renormalized solution to the continuity equation, this allowed in particular to treat the case of density-dependent viscosity in \cite{Des}.
Still in the framework of weak solutions, 
F. Fanelli and I. Gallagher investigated recently in \cite{FG} the fast rotation limit of \eqref{sys} supplemented with 
a Coriolis force.

Producing  `strong  solutions' (by strong, we mean solutions 
having  the uniqueness property) requires more constraints on the data: enough regularity and no vacuum, typically. 
Roughly speaking, according to the classical literature,
 for smooth enough data and provided
the density does not vanish,
we have  global existence of strong solutions  
\emph{even for  large data}  in dimension two, 
and, like for the constant density 
case, for small enough initial velocity in dimension three.
For such results in the bounded domain case, one can 
refer to the pioneering work by O.A. Ladyzhenskaya and 
V.~Solonnikov in \cite{LS} (further extended to less regular data 
by the  first author in \cite{D2}). 

A number of works have been dedicated to solving \eqref{sys} in  $\Omega=\R^d$ in so-called
`critical regularity frameworks'. The underlying idea
(that originates from Fujita and Kato's work \cite{FK}
for the constant density case) is that `optimal' 
functional spaces for well-posedness of \eqref{sys} 
have to share its scaling invariance, namely, for all 
$\ell>0,$
\begin{equation}\label{eq:critical}
(\varrho,v,P)(t,x) \leadsto (\varrho,\ell v,\ell^2P)(\ell^2t,\ell x)\andf (\varrho_0,v_0)(x) \leadsto (\varrho_0,\ell v_0)(\ell x).
\end{equation}
Observing that the couple of homogeneous Besov space
$\dot B^{\frac d2}_{2,1}(\R^d)\times
\bigl(\dot B^{\frac d2-1}_{2,1}(\R^d)\bigr)^d$ 
indeed possesses this invariance, the first author proved in
 \cite{D0} the well-posedness of \eqref{sys} supplemented
 with initial velocity $v_0$  in $\dot B^{\frac d2-1}_{2,1}(\R^d)$
and initial density  $\varrho_0$  close 
to some positive constant in $\dot B^{\frac d2}_{2,1}(\R^d).$
Note that, owing to the embedding $\dot B^{\frac d2}_{2,1}(\R^d)\hookrightarrow \cC_b(\R^d),$
this forces the density to be continuous. 
Subsequent  improvements have been brought to this approach 
(see e.g. \cite{AP}) but, 
still, the density  has to be `almost' continuous. 
In particular, one cannot consider initial densities 
that have  a jump across an interface, even a smooth one. 
\medbreak
Toward considering less regular densities, a first 
breakthrough has been brought by the first two authors
in \cite{DM0} and \cite{DM5}:  
 taking advantage of Lagrangian coordinates (that will be
presented below), they 
established well-posedness results 
for densities that are possibly discontinuous 
along interfaces, provided the jump is small enough. 

Then, in \cite{PZZ}, by a totally different approach, M. Paicu, 
P. Zhang
and Z. Zhang succeeded in proving the global existence in $\R^2$ for $v_0 \in H^s, \; s>0$ and in $\R^3$ for $v_0 \in H^1$ with $\|v_0\|_2\|\nabla v_0\|_2$ sufficiently small, 
provided the initial density satisfies
$$0<c_0\leq\vr_0\leq C_0<\infty.$$
In dimension $3,$ this work was extended in \cite{CZZ} to initial 
velocities that are only in $H^s$ for some $s>1/2.$
Still, the density  has to be  bounded  away from zero, and the solution in not time continuous with values in $H^s.$
Very recently, P. Zhang in \cite{Zhang20} achieved the critical regularity  $\dot B^{1/2}_{2,1}$ for the initial velocity,
but did not address the uniqueness issue in this setting. 

For more results where the initial density is allowed 
to be discontinuous but, still, strictly positive, 
the reader may  refer among others to \cite{GGJ,HPZ,FZ} and to a recent result \cite{CJ}, where an inflow boundary condition is considered. Let us also mention that global well-posedness in the half-space $\R^d_+$ with initial density only bounded but close to a positive constant was shown in \cite{DZ1}. 
\medbreak
All the above results require the strict positivity of the initial density. 
To our knowledge, the existence of unique solutions in presence of 
vacuum has been  first proved in \cite{CK} for rather high
regularity of the initial density and velocity,
(namely $\vr_0 \in L^{3/2} \cap H^2$ and $u_0\in H^2$)
and provided the following compatibility condition is satisfied:  
\begin{equation} \label{comp}
-\Delta v_0+\nabla P_0=\sqrt{\vr_0}g \quad \textrm{for some} \; g \in L_2 \; {\rm and} \; P_0 \in H^1.
\end{equation}
Global existence of unique solutions in a 3D bounded domain or in  $\R^3$ under the same compatibility condition and smallness of $\|u_0\|_{\dot H^{1/2}}$ was shown in \cite{CHW}. 

Condition \eqref{comp} was removed in \cite{Li}, where local well-posedness in a bounded domain is shown, but still for sufficiently smooth initial density. Global existence in the whole space $\R^3$, again under sufficient regularity of initial density, was proved recently in \cite{HLL}. 
  \medbreak
An important place in the theory of \eqref{sys} is taken by 
the so-called {\em density patch problem}: assuming that
\begin{equation} \label{ic0}
    \vr_0=\alpha_1 \chi_{A_0} + \alpha_2 \chi_{\Omega \setminus A_0}
    \end{equation}
 for some nonnegative constants $\alpha_1, \alpha_2$ and a measurable set $A_0,$ can we say that $\varrho(t)$ has the same structure
 for all time, 
 with persistence of the regularity of the interface~?
 This problem seems to have been first raised 
 by P.-L. Lions in \cite{PLL} in the specific case 
 where   $\vr_0=\chi_{A_0}$ with $A_0 \in \R^2,$ and $\sqrt{\vr_0}u_0 \in L_2.$ The original question was whether for all time $\vr(t)=\chi_{A(t)}$
for some domain  $A(t)$ with the same regularity  as $A_0.$

It turns out that a positive answer is obtained for the $C^1$ regularity as a consequence of the works of the first two authors in 
\cite{DM0,DM5}  if  $\alpha_1, \alpha_2>0$  are close to each other.
Much more complete results have been obtained in the two-dimensional case in \cite{LZ1} (case $\alpha_1-\alpha_2$ small), and then in \cite{LZ2} for any positive constants. There, the
authors actually establish the persistence of high `striated' Sobolev regularity for the density. 
Similar results have been proved in the 3D case in \cite{LL}.
The propagation of striated regularity has been adapted to the case 
where the viscosity depends on the density in \cite{PZ2}. 
By a different approach, persistence of H\"older continuity 
of the interface if $\alpha_1,\alpha_2$ are close to each other was shown in \cite{DZ2}.
\smallbreak
Requiring that the initial density is away from zero
precludes to consider the original Lions' problem, namely
the case when  $\alpha_2=0$ in \eqref{ic0}.
Recently in \cite{DM1}, the first and second authors proved the 
well-posedness of \eqref{sys} for only bounded initial density 
\begin{equation} \label{rho0}
0 \leq \vr_0 \leq \vr^*    
\end{equation}
and initial velocity satisfying  
\begin{equation} \label{v0}
v_0 \in H^1_0(\Omega), \quad \div v_0 =0.    
\end{equation}
In the two-dimensional case, the solutions are global
 without any additional condition while,  in the 3D case, 
  $v_0$  has to satisfy some smallness condition
  (as the results of \cite{DM1} are of particular importance for our analysis, they will be recalled  precisely below). 
As a by-product, the authors obtained a positive answer to the Lions question in the case $\vr_0=\chi_{A_0}$: 
persistence of H\"older regularity $C^{1,\alpha}$ 
holds true for any $0<\alpha<1$ in 2D and $0<\alpha<\frac{1}{2}$ in 3D.
\smallbreak
However, the question concerning the stability of the solutions 
was left open in \cite{DM1}. In fact, if the density is bounded away from zero
then  the stability can be proved in the same way as uniqueness, but this is no longer the case if the initial density  is allowed to vanish
(this has to do with  the parabolic character of the momentum equation, which degenerates if the density vanishes). This typically happens if we consider 
the following model configuration: 
 the original density is  $\vr^{(org)}_0$ 
and the perturbation  is  $\vr^{(per)}_0,$ that is
\begin{equation}\label{ic2}
    \vr_0^{(org)}=\chi_{A(0)}\andf\vr_0^{(per)}=\chi_{A(0)} + \phi.
\end{equation}

\begin{figure}[h!] \label{model} 
\centering 
\includegraphics[width=0.55\textwidth]{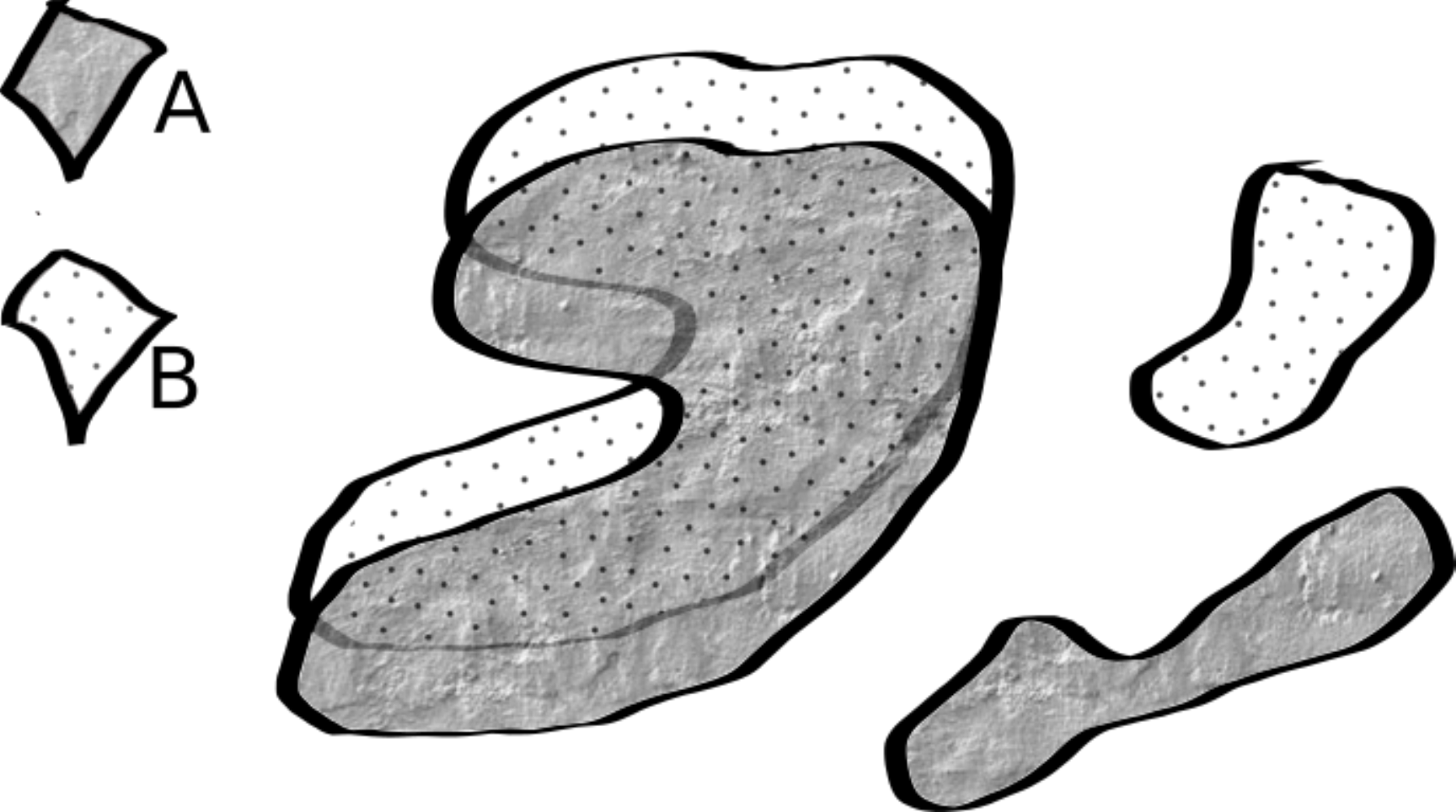}
\caption{Support of $\vr_0^{(org)} = {\bf A}$ and support of $\vr_0^{(per)}={\bf B}$.}
\end{figure} 
The problem happens whenever
\begin{equation*}
    (\supp \vr_0^{(org)} \setminus \supp \vr_0^{(per)} )\cup 
    (\supp \vr_0^{(per)} \setminus \supp \vr_0^{(org)})
    \ \hbox{ is not empty}.\end{equation*}
Indeed, unless $\phi=0,$ the perturbation is large in $L_\infty(\Omega)$
but small in $L_p(\Omega)$ for $p<\infty,$ and we need a functional framework for stability that captures this situation. 
\medbreak
The goal of the present paper 
is to show the stability of solutions to the inhomogeneous Navier-Stokes \eqref{sys} with respect to initial condition of type \eqref{ic}, in a regularity setting that includes the density patches problem \eqref{ic0} 
\emph{even if one of the parameters $\alpha_1,\,\alpha_2$ vanishes}.
In particular, we allow the density to be a characteristic function of some set. In the most pathological case the supports of $\vr_0^{(org)}$
and $\vr_0^{(per)}$ can be even disjoint.

\medbreak
Before stating our main stability result, we have to recall 
the state-of-the-art concerning the global well-posedness 
for  \eqref{sys} supplemented with general data satisfying \eqref{rho0} and \eqref{v0}.

In dimension $d=2,$   \cite[Thm 2.1]{DM1} states the following result:
\begin{thm} \label{thm:d=2}
Let $\Omega$ be a smooth bounded domain of $\R^2,$ or a two-dimensional torus. Let $\vr_0 \in L_\infty(\Omega)$ satisfy \eqref{rho0} and let $v_0$ satisfy \eqref{v0}. Then, System \eqref{sys} admits a unique solution $(\vr,v,P)$ such that 
\begin{equation}\label{reg:1}
\begin{aligned}
&\vr\in L_\infty(\R_+;L_\infty),\quad  v \in L_\infty(\R_+;H^1), \\ &\sqrt{\vr}v_t, \nabla^2 v,\nabla P \in L_2(\R_+;L_2), \quad
\nabla v\in L_{1,loc}(\R_+;L_\infty),\\
&\sqrt\vr v\in \cC(\R_+;L_2)\andf \vr\in\cC(\R_+;L_p) \quad \hbox{for all}\ 1\leq p <\infty.
\end{aligned}
\end{equation}
For arbitrarily large but finite time $T>0$ these solutions satisfy
in addition for all $1 \leq r <2, \; 1 \leq m<\infty, s<{1}/{2}$ and $1 \leq p <\infty,$
\begin{equation} \label{reg:2d}
\begin{aligned}
&\nabla(\sqrt{t}P), \nabla^2(\sqrt{t}v) \in L_{\infty}(0,T;L_r(\Omega)) \cap L_2(0,T;L_m(\Omega)), \; v \in H^s(0,T;L_p(\Omega)),\\
& \sqrt{t \vr} \,v_t \in L_\infty(0,T;L_2) \mbox{ \ and \ } \nabla v_t \in L_2(0,T;L_2).
\end{aligned}
\end{equation}
\end{thm}
\noindent
In the three dimensional case, we know from  \cite[Thm 2.2]{DM1} the following result:
\begin{thm} \label{thm:d=3}
Let $\Omega$ be a smooth bounded domain of $\R^3,$ or a three-dimensional torus. Let $\vr_0 \in L_\infty(\Omega)$ satisfy \eqref{rho0} and $v_0$ satisfy \eqref{v0}.
There exists $c>0$ such that if, in addition, 
\begin{equation} \label{small1}
(\vr^*)^{3/2}\|\sqrt{\vr_0}v_0\|_2\|\nabla v_0\|_2 \leq c\mu^2,
\end{equation}
then System \eqref{sys} admits a unique solution $(\vr,v,P)$ satisfying \eqref{reg:1} and, for any finite $T>0$, 
$s<{1}/{2}$ and $1 \leq p <\infty,$
\begin{equation} \label{reg:3d}
\begin{aligned}
&\nabla(\sqrt{t}P), \nabla^2(\sqrt{t}v) \in L_{\infty}(0,T;L_2(\Omega)) \cap L_2(0,T;L_6(\Omega)), \; v \in H^s(0,T;L_p(\Omega)), \\
& \sqrt{t \vr} \,v_t \in L_\infty(0,T;L_2) \mbox{ \ and \ } \nabla v_t \in L_2(0,T;L_2).
\end{aligned}
\end{equation}
\end{thm}
Although the above solutions are  unique, the question of their stability remains open so far. 
Here we aim at supplementing the above statements  
with a stability result. 
In order to obtain the most accurate information, 
it is natural  to use Lagrangian coordinates
since, in this setting,  the density is time independent (it only depends on the position of the particles initially). 
Therefore,  the problem is reduced to the control of the difference of the velocities
which, somehow, satisfies a parabolic equation. 
\medbreak
Let us shortly recall how to define  Lagrangian coordinates
in our setting. First, we introduce the flow 
$X:(t,y)\mapsto X(t,y)$ of $v,$
that is  the unique solution to the following ODE:
\begin{equation}\label{L-E}
\left\{    \begin{array}{lr}
   \displaystyle \frac{dX}{dt}=v(t,X(t,y))     & \mbox{in \ }  \R_+\times \Omega, \\[10pt]
   \displaystyle X(0,y)=y     &  \mbox{in \ } \Omega.
    \end{array}\right.
\end{equation}
Integrating \eqref{L-E} yields the following relation between the Eulerian $``x"$ and Lagrangian $``y"$ coordinates:
\begin{equation} \label{Lag1}
    x= X(t,y)=y+\int_0^t v(t',X(t',y))\dt'.
\end{equation}
By the standard theory of ODEs, the above change of coordinates
is well defined 
whenever $v \in L_{1,(loc)}(\R_+;C^{1,0})$. In the coordinates system $(t,y),$  the unknown functions are named as follows:
\begin{equation*}
    u(t,y)=v(t,X(t,y)), \;\;\; \eta(t,y)=\vr(t,X(t,y)), \;\;\; Q(t,y)=P(t,X(t,y)).
\end{equation*}
Let us denote 
\begin{equation} \label{Au}
A_u(t)=\left(\frac{dX}{dy}\right)^{-1}=\left(\Id+\int_0^t \nabla_y u(t',y)\dt'\right)^{-1}=\left[{\rm cof}\left(\Id+\int_0^t \nabla_y u(t',y)\dt' \right)\right]^T,    
\end{equation}
where $\rm cof(\cdot)$ denotes the cofactor matrix and, in the second equality, we used the fact that ${\rm det}A_u=1$ (see e.g. \cite{DM1}).
For a function $f(t,x),$ denote $\tilde f(t,y)=f(t,X(t,y))$. 
Then, owing to the chain rule,
\begin{equation} \label{trans1}
\nabla_x f(t,x)=A_u^T\nabla_y \tilde f(t,y)=:\nabla_u \tilde f(t,y), \quad 
\de_t f(t,x)+v(t,x)\cdot\nabla_xf(t,x)= \de_t \tilde f(t,y).
\end{equation}
In order to transform the divergence operator, observe that
$\divy A_u = 0$ since $\det A_u=1$ .
Therefore, for any vector field $z(t,y)$, one may write
\begin{equation} \label{trans:div1}
\divy(A_u z)=A^T : \nabla_y z + z \cdot \divy(A_u)=A^T : \nabla_y z.
\end{equation}
Hence, if we denote $\tilde w(t,y)=w(t,X(t,y))$
 for any vector field $w(t,x),$ we discover that
\begin{equation} \label{div:id}
\divx w(t,x)=A_u : \nabla_y \tilde w(t,y)=\divy(A_u \tilde w)=:\divu \tilde w.
\end{equation}
Taking all above into account we see that in coordinates $(t,y)$ 
system \eqref{sys} reads
\begin{equation} \label{sys:lag}
\left\{\begin{array}{llr}
    \eta_t=0     & \mbox{in}& \R_+\times\Omega,  \\[1ex]
    \eta u_t -\divu\nabla_u u + \nabla_u Q =0 &\mbox{in}& \R_+\times\Omega,\\   [1ex] 
    \divu u=0 &\mbox{in}& \R_+\times\Omega,\\[1ex]
    u|_{t=0}=v_0, \;\; \eta|_{t=0}=\vr_0 &\mbox{in}& \Omega.
    \end{array}\right.
\end{equation}
The main achievement of this paper  is the
following stability result in the Lagrangian coordinates setting. 
\begin{thm} \label{thm:lag}
Let $\Omega$ be a smooth bounded domain of $\R^d$ with $d=2,3$.
Let $(\vr^1,v^1)$ and $(\vr^2,v^2)$ be two solutions of \eqref{sys} 
with initial data $(\vr_0^1,v_0^1)$ and $(\vr_0^2,v_0^2)$, respectively,  with non identically zero bounded
 $\vr_0^1$ and $\vr_0^2,$ and $v_0^1,$ $v_0^2$ in $H^1_0(\Omega),$
 given either by Theorem \ref{thm:d=2} or by Theorem \ref{thm:d=3} (depending on the dimension). 
Denote by $u^1$ and $u^2$ the 
corresponding velocities in Lagrangian coordinates. 
Finally, set $\du:=u^2-u^1,$ $\dv:=v^2-v^1$ and 
$\dvr:=\vr^2-\vr^1.$

Then, there exists a positive constant $\beta$ depending
only on the shape of $\Omega,$ and such that 
\begin{equation} \label{decay:final}
\sup\limits_{t\in\R_+} \Bigl(e^{\frac{\beta\mu t} {\vr^*\delta^2}}[\|{\rm \min}\Big\{\sqrt{\vr^1_0},\sqrt{\vr^2_0}\Big\}\du(t)\|_2\Bigr) 
+ \|e^{\frac{\beta\mu t} {\vr^*\delta^2}}\nabla \du\|_{L_2(\R_+,L_2)}\leq C \left( \|\sqrt{\vr^1_0}\,\dv_0\|_2 + \|\dvr_0\|_2^{1/2} \right), 
\end{equation}
where $\delta$ stands for the diameter of $\Omega,$
and $\vr^*:=\|\varrho_0\|_{\infty}.$
\end{thm}

Coming back from Lagrangian to Eulerian coordinates, we obtain

\begin{cor} \label{cor:eul}
Under the assumptions of Theorem \ref{thm:lag}, in Eulerian coordinates, we have 
 $$  \sup_{t\in\R_+} \|\dvr(t)\|_{W^{-1}_p} 
   + \|e^{\frac{\beta\mu t} {\vr^*\delta^2}}\nabla \dv\|_{L_2(\R_+,L_2)}
   \leq C_0 \left( 
    \|\sqrt{\vr^1_0}\,\dv_0\|_2 + \|\dvr_0\|_2^{1/2}\right)$$
for $1<p<\infty$ if $d=2$ and $1<p\leq 6$ if $d=3$.    
\end{cor} 
In our low regularity setting, the key difficulty for proving 
stability of solutions to \eqref{sys} comes from the partially 
hyperbolic nature of the system. In fact,  if writing 
the system satisfied by the difference 
$(\dvr,\dv,dP)$ of two solutions $(\varrho^1,v^1,P^1)$
and $(\varrho^2,v^2,P^2)$ of \eqref{sys}, then  the mass equation gives:
$$(\partial_t+v^2\cdot\nabla)\dv=\dv\cdot\nabla\vr^2.$$
In our framework where the density is only in $L_\infty(\Omega),$
this forces us to perform estimates for $\dvr$ in 
a space with regularity index equal to $-1.$
Following the duality approach initiated by D. Hoff in \cite{Hoff} and recently
renewed in \cite{DM2} and \cite{MP} (for
the related compressible Navier-Stokes system), we shall 
actually prove stability estimates for
the density in  $W^{-1}_p,$  and in $L_2$ for the velocity. 

A key point in order to get rid of any smallness condition, is to first establish a sufficiently strong time-decay of the solutions that have been constructed in Theorems \ref{thm:d=2} and \ref{thm:d=3}.  This will be achieved
in Section \ref{s:decay}, where by means of rather classical energy arguments, we will even get \emph{exponential} decay,
  owing to the boundedness of the fluid domain. 
Before that, in Section \ref{s:stability}, we shall  compare two global solutions satisfying {\em a priori} those decay properties, and estimate their difference in Lagrangian coordinates
in terms of the difference of the data, getting the result of Theorem \ref{thm:lag}. Finally, we rewrite those estimates in the Eulerian coordinates to obtain Corollary \ref{cor:eul}.

Proving the exponential decay estimates of  Section \ref{s:decay}
 can be achieved by means of a remarkably simple  energy method
 that is performed directly on the nonlinear problem. 
 This is in sharp contrast with the proof of decay estimates
 for compressible Navier-Stokes and related models
 which requires a refined analysis of the linearized 
 system combined with a perturbation argument
 (see among others \cite{DHP,ShiShi,ES,PSZ} and the references therein).
 
%
%
When  comparing  Theorem \ref{thm:lag} with results of \cite{DM1}, a remark is in order concerning the domain. Although Theorems \ref{thm:d=2} and \ref{thm:d=3} hold both for a bounded domain with Dirichlet condition and a torus, we here restricted our analysis to the case of a bounded domain
to  avoid further technical complications. In fact, in  the case of Dirichlet boundary conditions, 
we have  the  basic Poincar\'e inequality at our disposal, which is fundamental  to close the estimates
globally in time. In the torus case,  the corresponding 
 Poincar\'e inequality has an additional term (namely the total momentum of the solution, see Lemma A.1 in \cite{DM1})
 which, although probably harmless, entails serious complications in the 
 proof of  decay estimates. Therefore, we leave the torus case for future research. 
\smallbreak

{\bf Notation.} We use standard notations $L_p$ and $W^k_p$ for Lebesque and Sobolev spaces (the dependency 
with respect to the fluid domain $\Omega$ is omitted). 
For the corresponding norms, we use the short notation: 
$$
\|\cdot\|_p := \|\cdot\|_{L_p}, \qquad \|\cdot\|_{k,p} := \|\cdot\|_{W^k_p}.  $$
In some computations, we will agree
that  $f_p(t)$ denotes a generic function of time which is in $L_p(\R_+)$, and  that  $f_{p,q}(t)$ stands for a function which is in $L_p(\R_+) \cap L_q(\R_+)$. The precise form of these functions may vary from line to line, but the property of
integrability is preserved.


\section{Stability under given decay properties}\label{s:stability}

Throughout this section, we are given two solutions 
$(\vr^1,v^1,P^1)$ and $(\vr^2,v^2,P^2)$ pertaining to data  $(\vr_0^1,v_0^1)$ and $(\vr_0^2,v_0^2),$ and satisfying 
the following properties for some $\beta_0>0$:

\begin{subequations}\label{eq:regularity}
\begin{align}
&\sqrt{\vr}\, e^{\beta_0 t} v_t^i \in L_2(\R_+,L_2), \label{eq:regularity:8}\\[3pt]
& e^{\beta_0 t} \nabla v^i\in L_1(\R_+;L_\infty)\cap L_4(\R_+;L_3)\cap L_2(\R_+;L_6),\label{eq:regularity:2}\\[3pt]
&\sqrt t\, e^{\beta_0 t} v^i\in L_1\cap L_\infty(\R_+;L_\infty),\label{eq:regularity:3}\\[3pt]
&\sqrt t\, e^{\beta_0 t}(\nabla^2v^i,\nabla  P^i)\in L_2(\R_+;L_6),\label{eq:regularity:4}\\[3pt]
&\sqrt t\,e^{\beta_0 t} v_t^i \in L_{4/3}(\R_+;L_6), \label{eq:regularity:5}\\[3pt]
& e^{\beta_0 t} v^i  \in L_1\cap L_\infty(\R_+;L_6), \label{eq:regularity:6}\\[3pt]
&\sqrt t\, e^{\beta_0 t} \nabla v^i\in L_2(\R_+;L_\infty),\label{eq:regularity:1}\\[3pt]
& e^{\beta_0 t} \nabla^2  v^i \in L_1(\R_+;L_r) \; \textrm{with} \; r>d. \label{eq:regularity:7}
\end{align}
\end{subequations} 
\noindent
We denote by $(\eta^i,u^i,Q^i)$ the corresponding solutions in Lagrangian coordinates (hence $\eta^i=\vr_0^i$). 

\begin{lem} \label{l:est_lag}
Let $(\vr,v,P)$ solve \eqref{sys} and satisfy Conditions \eqref{eq:regularity:8} to \eqref{eq:regularity:7}.
 Then, the corresponding Lagrangian solution  $(\eta,u,Q)$ also satisfies  \eqref{eq:regularity:8} to \eqref{eq:regularity:7}.
\end{lem}
\begin{proof} 
All  properties except for the ones involving a time derivative 
and the last one just follow from  
the corresponding ones  for $v,$ and from the fact that  the matrix $A_u$ is bounded.

For proving  \eqref{eq:regularity:8}, we start from the identity 
$$\sqrt\eta e^{\beta_0t} u_t=\sqrt{\rho\circ X_t}\,e^{\beta_0t}(v_t+v\cdot \nabla v)\circ X_t.
$$
As $X_t$ is measure preserving, the term  with $v_t$ may be bounded by means of 
\eqref{eq:regularity:8}. For bounding the other term, it suffices to observe that
$$\|e^{\beta_0t}\sqrt{\rho\circ X_t}\,v\cdot \nabla v)\circ X_t\|_{L_2(\R_+\times\Omega)}
\leq \rho^*\|e^{\beta_0t/2} v\|_{L_\infty(\R_+;L_6)}\|e^{\beta_0t/2}\nabla v\|_{L_2(\R_+;L_3)}.$$
The first term of the right-hand side may be bounded thanks to \eqref{eq:regularity:6}
while the second one, according to \eqref{eq:regularity:2} and the boundedness of $\Omega.$ 
As regards  \eqref{eq:regularity:5},  one can use again
$u_t=(v_t+v\cdot \nabla v)\circ X_t,$ and  properties    \eqref{eq:regularity:5} to  \eqref{eq:regularity:1} for $v.$  

In order to prove \eqref{eq:regularity:7}, we differentiate the identity $\nabla_y u(t,y)={}^T\!(A_u)^{-1}\nabla_x v(t,X(t,y))$ with respect to $y$. By \eqref{Au} we obtain $$
\|\nabla^2_y u\|_{L_1(0,T;L_r)}\leq C\|\nabla_x^2 v\|_{L_1(0,T;L_r)} +
C\|\nabla_y A\|_{L_\infty(0,T;L_r)}\|\nabla_y u\|_{L_1(0,T;L_\infty)},
$$
which implies \eqref{eq:regularity:7} for $u$ due to  \eqref{eq:regularity:7} for $v$  and to \eqref{eq:regularity:2} for $u$.
\end{proof}

In the rest of this section, we aim at estimating
$$
\du:= u^2-u^1 \mbox{ \ \  and  \ \ } \dQ:= Q^2-Q^1
$$
in terms of the difference of the data.  Obviously, denoting $\Delta_u:=\divu\nabla_u$ and $
\dv_0=v^2_0-v^1_0$, the couple $(\du,\dQ)$ satisfies
\begin{equation}\label{NSL-uniq}
 \begin{aligned}
  &{\vr_0^1}\du_t - \Delta_{u^1}\du + \nabla_{u^1}\dQ= 
(\Delta_{u^2} -\Delta_{u^1}) u^2 - (\nabla_{u^2}-\nabla_{u^1}) Q^2-\dvr_0 \,u^2_t,\\
&\div_{\!u^1}\du=(\div_{\!u^1}-\div_{\!u^2})u^2,\\
&\du|_{t=0}=\dv_0.
\end{aligned}
\end{equation}
Note that
\begin{equation*}
    \dvr_0:=\vr_0^2 - \vr_0^1=\eta^2-\eta^1.
\end{equation*}
By $(\ref{sys:lag})_1,$ functions $\eta^i$ are constant in time, so the perturbation of the density stays the same in time. It is one of the main reasons why we chose the Lagrange coordinates approach to deal with the 
stability issue of system \eqref{sys}. 

\subsection{The case of a nice control of vacuum}

Compared to the proof of uniqueness that has been performed in \cite{DM1}, the troublemaker is 
 the term $\dvr_0 \,u^2_t$    in the equation for $\du$ since  Theorems \ref{thm:d=2}
 or \ref{thm:d=3} only provide us with an information on $\sqrt{\vr^2}\, v_t^2$ (hence on $\sqrt{\vr_0^2}\,u_t^2$) 
while we do not  necessarily have  
 \begin{equation}\label{x3}
 \supp \dvr_0 \subset \supp \vr_0^2.
\end{equation}
In this part, we  assume that the initial densities satisfy 
\begin{equation} \label{def:xnorm}
\|\dvr_0\|_X \,:= \, \|{\dvr_0}/{\sqrt{\vr_0^2}}\|_4 < \infty   
\end{equation}
and  derive a differential inequality which is crucial for proving Theorem \ref{thm:lag}. 
The general case, when \eqref{x3} is not valid,  will be investigated in the  next subsection.  
\medbreak
The idea is to decompose $\du$ into
\begin{equation}\label{l5}
 \du = w+z,
\end{equation}
where $w$ stands for a suitable solution to the divergence equation
\begin{equation}\label{l6}
 \div_{\!u^1} w = (\div_{\!u^1}-\div_{\!u^2})u^2={}^T\!\dA: \nabla u^2=\div (\dA \, u^2),
\end{equation}
with $\dA:=A^1-A^2$ and $A^i:=A_{u^i}.$ 
\smallbreak
Since at $t=0$ by \eqref{Au} we have $\dA=0$, we put $w=0$ at $t=0$.
Note that $A^1$ and $A^2$ need not to be close to $\Id$, but  are invertible and uniformly bounded in time.

\begin{lem}\label{Lem:9}
There exists a solution to \eqref{l6} such that
\begin{equation} \label{est:w}
\begin{aligned}
& \|e^{\beta_0 t}w(t)\|_{2} \leq  f_{1,\infty}(t) \biggl(\int_0^t \|\nabla \du(\tau)\|_2^2 \dtau\biggr)^{1/2}, \\
& \|e^{\beta_0 t}\nabla w(t)\|_{2} \leq f_2(t)  \biggl(\int_0^t \| \nabla \du(\tau)\|_2^2\, d\tau\biggr)^{1/2}, \\
& \|e^{\beta_0 t}w_t(t)\|_{3/2} \leq f_{4/3}(t) \biggl(\int_0^t \| \nabla \du(\tau)\|_2^2\, d\tau\biggr)^{1/2}+
f_{4}(t)\|\nabla \du(t)\|_{2},
\end{aligned}
\end{equation}
where the notation $f_p(t)$ and $f_{p,q}(t)$ is explained at the end of Section \ref{sec:intro}. 
\end{lem}
\begin{proof}
The vector $w$ will be sought in the form 
$    w=(A^1)^{-1} A^1 w= (A^1)^{-1} \bar w $,
where  $\bar w$ is given as a solution to 
\begin{equation}\label{l7}
 \div \bar w \big(= \div (A^1 w) = \div_{u^1} w\big) = (\div_{\!u^1}-\div_{\!u^2})u^2= {}^T\!\dA: \nabla u^2 =\div (\dA \, u^2).
\end{equation}

\medbreak
As a first step in the proof of our claim, let us establish the following bounds:
\begin{equation}\label{eq:www}\begin{aligned}
\|e^{\beta_0 t}\bar w\|_{L_4(0,T;L_2)}&\leq C\|e^{\beta_0 t}\dA \,u^2\|_{L_4(0,T;L_2)}, \\
\|e^{\beta_0 t}\nabla \bar w\|_{L_2(0,T;L_2)}&\leq C\|e^{\beta_0 t}\;{}^T\!\dA:\nabla u^2\|_{L_2(0,T;L_2)}\\
\hbox{\rm and }\ 
\|e^{\beta_0 t}\bar w_t\|_{L_{4/3}(0,T;L_{3/2})}&\leq C\|e^{\beta_0 t}(\dA \,u^2)_t\|_{L_{4/3}(0,T;L_{3/2})}.
\end{aligned}\end{equation}
The existence of a vector field $\bar w$ satisfying \eqref{l7}-\eqref{eq:www} is assured by the following Lemma:
\begin{lem} \label{lem:Bog}  
Let $A$ be a matrix valued function with $\det A\equiv1.$
Consider the following divergence equation in a bounded domain with smooth boundary:
\begin{equation*}
    \div b =f \mbox{ \ \ in \ }  \R_+\times\Omega,  \qquad b=0 \mbox{ \ \ at \ }  \R_+\times\partial\Omega,
\end{equation*}
where $f=A^T:\nabla d=\div(Ad)$ and the average of $f$ is equal zero. 

Then, there exists a  constant $C$  such that for any $\beta\geq0,$
there exists a solution $b$  to the above equation, such that for all 
$t\geq0,$ we have
\begin{equation*}
\begin{aligned}
&\|e^{\beta t} b(t)\|_2 \leq C\|e^{\beta t} A(t) d(t)\|_2, \qquad \|e^{\beta t}\nabla b(t)\|_2 \leq C \|e^{\beta t}A(t)^T \nabla d(t)\|_2,\\[5pt] 
&\|e^{\beta t}b_t(t)\|_{3/2} \leq C\|e^{\beta t}(Ad)_t(t)\|_{3/2}.
\end{aligned}
\end{equation*}
\end{lem}
Lemma \ref{lem:Bog} has been proved without exponential weight by the first two authors in \cite{DM3} (see also
\cite[Lemma A.3]{DM1}). In their proof the function $b$ is given by an explicit formula, so a time weight can be treated as a multiplicative parameter (as the norms in Lemma \ref{lem:Bog} only involve
the space variable).
\medbreak
Now, let us bound the right-hand sides of \eqref{eq:www}. 
In order to emphasize that we do not need any smallness of $\int_0^t\nabla_yu\,\dtau,$
let us derive an explicit formula for $\dA.$

In the two dimensional case, starting from \eqref{Au},  we immediately obtain 
\begin{equation} \label{dA:2d}
\dA(t) = \left[\begin{array}{cc} 
\int_0^t \du_{2,y_2}\dt' & -\int_0^t \du_{1,y_2}\dt'\\[1ex]
-\int_0^t \du_{2,y_1}\dt' & \int_0^t \du_{1,y_1}\dt' 
\end{array}
\right]\cdotp
\end{equation}
In three dimensions we can also use \eqref{Au}, let us compute precisely one entry of $\dA$. We have 
$$
(A_u^i)_{11}(t)=\biggl(1+\int_0^t u^i_{2,y_2}\dt'\biggr)\biggl(1+\int_0^t u^i_{3,y_3}\dt'\biggr)-
\int_0^t u^i_{3,y_2}\dt'\:\int_0^t u^i_{2,y_3}\dt'
$$
for $i=1,2$, therefore 
$$\displaylines{\quad
(\dA(t))_{11}= \int_0^t \du_{2,y_2}\dt' + \int_0^t \du_{3,y_3}\dt' 
+\int_0^t \du_{2,y_2}\dt'\int_0^t u^2_{3,y_3}\dt'+\int_0^t \du_{3,y_3}\dt'\int_0^t u^1_{2,y_2}\dt'
\hfill\cr\hfill
-\int_0^t \du_{3,y_2}\dt'\int_0^t u^2_{2,y_3}\dt'
-\int_0^t \du_{2,y_3}\dt'\int_0^t u^1_{3,y_2}\dt'.\quad}$$
Other entries will have a similar structure, we can write them as
\begin{equation} \label{dA:3d}
(\dA(t))_{ij} = \sum_{1 \leq k,l \leq 3} a^{ij}_{kl}\int_0^t\du_{k,y_l}\dt' 
+ \sum_{1 \leq k,l,m,n \leq 3, \; s\in\{1,2\} }b^{ij}_{k,l,m,n,s}\int_0^t \du_{k,y_l}\dt'\:
\int_0^t u^s_{m,y_n}\dt',
\end{equation}
where $a^{ij}_{kl},b^{ij}_{k,l,m,n,s} \in \{0,1\}$.
Now if $u^1,u^2$ satisfy \eqref{eq:regularity:2} then, by H\"older inequality, we obtain 
\begin{equation}\label{l128}
  \|t^{-1/2} \dA(t)\|_2 \leq C\Big\|t^{-1/2}\int_0^t \nabla \du\dtau\Bigr\|_{2}
  \leq C\biggl(\int_0^t\|\nabla\du\|_2^2\dtau\biggr)^{1/2}.
\end{equation}
Therefore, by \eqref{eq:regularity:1}, we have  
\begin{equation}\label{eq:new}
\|e^{\beta_0 t}\;{}^T\!\dA:\nabla u^2(t)\|_2\leq 
   \|t^{-1/2} \dA(t)\|_2\|e^{\beta_0 t}t^{1/2}\nabla u^2\|_{\infty} \leq  C
   f_2(t) \left(\int_0^t\|\nabla\du\|_2^2\dtau\right)^{1/2}, 
 \end{equation}
 which implies \eqref{est:w} for $\|\nabla \bar w\|_{L_2}$.
Similarly, by \eqref{eq:regularity:3} and \eqref{l128}
\begin{align*}
\|e^{\beta_0 t}\dA\, u^2\|_2
\leq \|t^{-1/2}\dA\|_2\|e^{\beta_0 t}t^{1/2}u^2\|_{\infty}
\leq f_{1,\infty}(t) \biggl(\int_0^t\|\nabla\du(\tau)\|_2^2\dtau\biggr)^{1/2},  
\end{align*}
whence, applying \eqref{eq:www} gives
\begin{equation}\label{l9} \|e^{\beta_0 t} \bar w(t)\|_2 
  \leq f_1(t) \biggl(\int_0^t\|\nabla\du(\tau)\|_2^2\dtau\biggr)^{1/2}. 
\end{equation}
In order to bound $\bar w_t,$ it suffices to derive an appropriate estimate 
in $L_{4/3}(0,T;L_{3/2})$ for 
$$(\dA\,u^2)_t=\dA\,u^2_t+(\dA)_t\,u^2.$$
For the first term, thanks to \eqref{eq:regularity:5} and \eqref{l128} we have
$$
\begin{aligned}
\|e^{\beta_0 t}\dA\,u^2_t\|_{3/2}&\leq \|t^{-1/2}\dA\|_2\|e^{\beta_0 t}t^{1/2}u^2_t\|_6
\leq f_{4/3}(t)\biggl(\int_0^t\|\nabla\du(\tau)\|_2^2\dtau\biggr)^{1/2}. 
\end{aligned}
$$
The other term can be bounded as follows: 
$$\|e^{\beta_0 t}(\dA)_t\,u^2\|_{3/2}\leq
\|(\dA)_t\|_2\|e^{\beta_0 t}u^2\|_6.$$
Differentiating \eqref{dA:3d} with respect to $t$ and using \eqref{eq:regularity:1} for $u^1$ and $u^2,$ 
we see that
$$
\|\dA_t(t)\|_{2}\leq C\biggl(\|\nabla\du(t)\|_{2}+\Big\|t^{-1/2}\int_0^t \nabla \du(\tau)\dtau\Big\|_2
\bigl(\|t^{1/2}\nabla u^1(t)\|_\infty+\|t^{1/2}\nabla u^2(t)\|_\infty\bigr)\biggr).
$$
In the 2D case, owing to \eqref{dA:2d}, one can skip the second term on the right-hand side
of the above inequality. Thus, 
owing to \eqref{eq:regularity:6}, for both $d=2,3$ we obtain
$$
\|\dA_t\,u^2(t)\|_{3/2}\leq f_{4/3}(t) \biggl(\int_0^t\|\nabla\du(\tau)\|_2^2\dtau\biggr)^{1/2}+
f_4(t)\|\nabla \du(t)\|_{2}, 
$$
which, by \eqref{eq:www} implies \eqref{est:w} for $\|\bar w_t\|_{3/2}$.
Altogether, this gives the thesis of Lemma \ref{Lem:9}
for $\bar w$, but not yet for $w$. 
\smallbreak

In order to get \eqref{est:w} for $w$ from the estimates of $\bar w,$ we 
first observe that 
\begin{equation} \label{A1}
\sup_{t\leq T}\|(A_u)^{-1}(t)\|_\infty \leq C \|\nabla u\|_{L_1(0,T;L_\infty)}.
\end{equation}
Therefore
$$    \|e^{\beta_0 t}w\|_{L_4(0,T;L_2)} \leq C\|A^1\|_{\infty}\|e^{\beta_0 t}\bar w \|_{L_4(0,T;L_2)} \leq C \|e^{\beta_0 t}\bar w \|_{L_4(0,T;L_2)}. $$
In order to estimate $\|\nabla w\|_{L^2(0,T;L^2)}$ we proceed as follows:
\begin{align*}
\|e^{\beta_0 t}\bar w \nabla((A^1)^{-1})\|_{L_2(0,T;L_2)} & \leq 
\|\nabla((A^1)^{-1})\|_{L_\infty(0,T;L_r)}\|e^{\beta_0 t}\bar w\|_{L_2(0,T;L_{r*})}\\
& \leq C \|\nabla^2 u^1\|_{L_1(0,T;L_r)}\|e^{\beta_0 t}\bar w\|_{L_2(0,T;L_{r*})}
\end{align*}
with $\frac12 =\frac{1}{r} + \frac{1}{r^*}$, where in the last passage we used \eqref{Au}. Therefore by \eqref{eq:regularity:2} and \eqref{eq:regularity:7} 
$$
\begin{aligned}
    \|e^{\beta_0 t}\nabla w\|_{L_2(0,T;L_2)} &\leq \|e^{\beta_0 t}(A^1)^{-1} \nabla\bar w\|_{L_2(0,T;L_2)} + \|e^{\beta_0 t}\bar w \nabla((A^1)^{-1})\|_{L_2(0,T;L_2)}\\
    & \leq C [\|e^{\beta_0 t}\nabla \bar w\|_{L_2(0,T;L_2)} + \|e^{\beta_0 t} \bar w\|_{L_2(0,T;L_{r^*})} ]. 
\end{aligned}$$
\noindent
Since, in  \eqref{eq:regularity:7}, one can take $r>d$,  one can always ensure that  $r^*<6.$ 
\smallbreak
Hence, Finally we have
$$
\|(A^1)^{-1})_t \bar w\|_{3/2} \leq \|(A^{-1})_t\|_6 \|\bar w\|_2
\leq \|\nabla u^1\|_6 \|\bar w\|_2,
$$
which together with \eqref{A1} implies
$$\begin{aligned}
\|e^{\beta_0 t}w_t\|_{L_{4/3}(0,T;L_{3/2})}
&\leq C\|(A^1)^{-1}\|_{\infty} \|e^{\beta_0 t}\bar w_t\|_{L_{4/3}(0,T;L_{3/2})} +
    C\|\nabla u^1\|_{L_2(0,T;L_6)} \|e^{\beta_0 t}\bar w\|_{L_4(0,T;L_2)}\\
&\leq C \bigl(\|e^{\beta_0 t}\bar w_t\|_{L_{4/3}(0,T;L_{3/2})} + \|e^{\beta_0 t}\bar w\|_{L_4(0,T;L_2)}\bigr),
\end{aligned}$$
which completes the proof of Lemma \ref{Lem:9}.\end{proof} 

\smallbreak
Next, let us restate the equations \eqref{NSL-uniq} as the following system for $(z,\dQ)$: 
\begin{equation}\label{l21}
 \left\{\begin{aligned}
&{\vr^1_0} z_t -\Delta_{u^1} z + \nabla_{u^1}\dQ = 
 (\Delta_{u^2}-\Delta_{u^1})u^2+(\nabla_{u^1}-\nabla_{u^2})Q^2
   - {\vr^1_0} w_t + \Delta_{u^1} w - \dvr_0 \, u^2_t, \\[1ex]
&\div_{\!u^1} z=0, \qquad z|_{t=0}=\dv_0.
 \end{aligned}\right.
\end{equation}
Observe that for a vector field $z$ and functions $f,g$ defined in Lagrangian coordinates we have,
according to \eqref{trans1} and integration by parts, 
\begin{equation} \label{parts:div}
\begin{aligned}
-\int_{\Omega}f \divu z \dy & = -\int_{\Omega}f\div (A_u z)\dy
=\int_{\Omega}A_u z \cdot \nabla_y f\dy\\[5pt] 
& = \int_{\Omega} z \cdot  A_u^T\nabla_yf \dy = \int_{\Omega} z \cdot \nabla_u f\dy, 
\end{aligned}
\end{equation}
which implies
\begin{equation} \label{parts:lap}
-\int_{\Omega} f \Delta_u g \dy = -\int_{\Omega} f \divu (\nabla_u g) \dy 
= \int_{\Omega} \nabla_u f \cdot \nabla_u g \dy.
\end{equation}
These identities allow to test \eqref{l21} by $z$ while 
\eqref{parts:div} implies the following crucial property
thanks to which one does not have to care about the difference of the pressures:
\begin{equation}\label{l22}
 \int_{\Omega} (\nabla_{u^1} \dQ)\cdot z \dy = -\int_{\Omega} \div_{\!u^1} z \, \dQ\dy =0.
\end{equation}
Therefore, using also \eqref{parts:lap}, we obtain
\begin{equation}\label{l23}
 \frac 12 \frac{d}{dt} \int_{\Omega} {\vr_0^1} |z|^2 \dx + \int_{\Omega} |\nabla_{u^1} z|^2\dx = \sum_{k=1}^5 I_k,
\end{equation}
where
$$
\begin{array}{lll}
 &\displaystyle I_1:= \int_{\Omega} \bigl((\Delta_{u^2} -\Delta_{u^1}) u^2\bigr) \cdot z \dy,  \quad 
&  \displaystyle I_2:= \int_{\Omega}  ((\nabla_{u^1}-\nabla_{u^2})Q^2)\cdot z \dy,\\[3ex]
 &\displaystyle  I_3:= - \int_{\Omega} \, {\vr^1_0} \, w_t \cdot z \dy, \quad & \displaystyle  I_4: = \int_{\Omega} (\Delta_{u^1} w)\cdot z\dy,\\[3ex]
 & \displaystyle I_5:= -\int_{\Omega} \, \dvr_0\,u_t^2 \cdot \,z \dy. & 
\end{array}$$
%
In order to bound $I_1,$ we combine   \eqref{eq:regularity:1} and \eqref{l128} to write for all $\ep>0,$
\begin{eqnarray}
| I_1| &\!\!\!=\!\!\!& \left| \int_{\Omega} \div\bigl(((\dA)^T\!A_2+ (A_1^T)\,\dA)\nabla u^2\bigr)\cdot z\dy \right|
\nonumber\\
 &\!\!\!\leq\!\!\!&\int_{\Omega} \bigl|(\dA^T)\,A_2 + (A_1^T)\, \dA\bigr|\:  |\nabla u^2 |\: |\nabla z| \dx
 \nonumber\\
 &\!\!\!\leq\!\!\!& C\|t^{-1/2} \dA\|_2 \|t^{1/2} \nabla u^2\|_\infty \|\nabla z\|_2\nonumber\\\label{est:i1}
  &\!\!\!\leq\!\!\!&\varepsilon \|\nabla z\|^2_2 + C\ep^{-1}\|t^{1/2} \nabla u^2\|_\infty^2  \biggl(\int_0^t \|\nabla \du\|_2^2\dtau\biggr). 
\end{eqnarray}
Next, according to \eqref{eq:regularity:4},\eqref{l128} and to Sobolev embedding
$H^1_0\hookrightarrow L_4,$
\begin{equation*}
 |I_2(t)| \leq  \bigg|\int_{\Omega} \dA \, \nabla Q^2 \cdot z \dx \bigg|\leq C\|t^{-1/2} \dA\|_2 \|t^{1/2} \nabla Q^2\|_4 \|z\|_4, 
\end{equation*}
Therefore, for all $\ep>0,$
\begin{equation} \label{est:i2}
  I_2(t) 
\leq \ep\|\nabla z\|^2_2 + C\ep^{-1}\|t^{1/2} \nabla Q^2\|_4^2  \biggl(\int_0^t \|\nabla \du\|_2^2 \dtau\biggr)\cdotp
 \end{equation}
 Note that  from  H\"older inequality and the Sobolev embedding $H^1(\Omega)\hookrightarrow L_6(\Omega)$, we have
$$
 \|{\vr_0}^{1/4}z\|_{{3}}\leq \|\sqrt{\vr_0}\, z\|_{2}^{1/2}\|z\|_{6}^{1/2}
 \leq C \|\sqrt{\vr_0}\, z\|_{2}^{1/2}\|z\|_{H^1}^{1/2}.
 $$
 Therefore using H\"older inequality, one can write that
 \begin{equation} \label{I3:5}
 \begin{aligned}
  I_3(t)& \leq \|w_t\|_{3/2}| \|\vr_0^1 z\|_{3} 
  \\&\leq C \|w_t\|_{3/2}| \|(\vr_0^1)^{1/4} z\|_{3}  \\
  &\leq \|w_t\|_{3/2}\|\sqrt{\vr_0^1} z\|_2^{1/2}\| z\|^{1/2}_{H^1}.
  \end{aligned}
 \end{equation}
 %

%

 Next, integrating by parts,  we get for all $\ep>0,$ 
\begin{equation} \label{est:i4}
 I_4 \leq  \int_{\Omega} |\nabla_{u^1} w|\, | \nabla_{u^1} z| \dx 
  \leq    \ep \|\nabla_{u^1} z\|_2^2+ C\ep^{-1}\|\nabla_{u^1} w\|_2^2.
\end{equation}

Finally, in order to estimate $I_5,$ we apply 
H\"older  and Young inequality, as well as the embedding
$H^1_0\hookrightarrow L_4$ to get
\begin{align} \label{est:i5}
    I_5 &\leq \left| \int_{\Omega} \frac{\dvr_0}{\sqrt{\vr^2_0}} \, \sqrt{\vr_0^2} u^2_t\, z \dy \right|\nonumber \\
    &\leq
    \biggl\|\frac{\dvr_0}{\sqrt{\vr^2_0}}\biggr\|_{4} \|\sqrt{\vr^2_0} u_t^2\|_{2}\|z\|_4\nonumber\\ &\leq
    C\ep^{-1}\|\sqrt{\vr^2_0} u_t^2\|_2^2 \|\dvr_0\|^2_X + \ep  \|\nabla z\|_2^2,
\end{align}
where $\|\dvr_0\|_X$ is defined in \eqref{def:xnorm}. 
Let us point out that our `nice control of vacuum' hypothesis \eqref{def:xnorm} comes
into play only for handling~$I_5.$
\medbreak
In the end, plugging \eqref{est:i1}, \eqref{est:i2}, \eqref{I3:5}, \eqref{est:i4} and \eqref{est:i5} in  \eqref{l23},  we obtain
$$
\begin{aligned}
\frac12\frac{d}{dt} \|\sqrt{{\vr_0^1}} z(t)\|_2^2 + \|\nabla_{u^1} z\|_2^2  \leq \; & \ep \|\nabla z\|_2^2   
+\left(\|t^{1/2}\nabla u^2\|_\infty^2+\|t^{1/2}\nabla Q^2\|_4^2\right) \int_0^t \|\nabla \du\|^2_2 \dt\\ 
&+ \|\sqrt{\vr_0^2}u_t^2\|_2^2 \|\dvr_0\|_X^2 
+C\|\nabla_{u_1}w\|_2^2
+\|w_t\|_{3/2}\|\sqrt{\vr_0^1}z\|_2^{1/2}\|z\|_{H^1}^{1/2}.
\end{aligned}
$$
Since $\nabla u_1={}^T\,A_{u^1}\nabla z,$ we have
$$\|\nabla_{u^1}z\|_2\geq\|(A_{u^1})^{-1}\|_\infty^{-1}\|\nabla z\|_2.$$
Hence, taking $\ep$ small enough, the above inequality implies 
for some $c_0>0,$
\begin{multline}\label{l33}
\frac{d}{dt} \|\sqrt{{\vr_0^1}} z(t)\|_2^2 + 
c_0\|\nabla z\|_2^2  
\leq  \left(\|t^{1/2}\nabla u^2\|_\infty^2+\|t^{1/2}\nabla Q^2\|_4^2\right) \int_0^t \|\nabla \du\|^2_2 \dt\\ 
+ \|\sqrt{\vr_0^2}u_t^2\|_2^2 \|\dvr_0\|_X^2 
+C\|\nabla_{u_1}w\|_2^2
+\|w_t\|_{3/2}\|\sqrt{\vr_0^1}z\|_2^{1/2}\|z\|_{H^1}^{1/2}. 
\end{multline}
Now, using the fact that $\|\nabla \du\|_2^2 \leq 2(\|\nabla z\|_2^2 + \|\nabla w\|_2^2),$ multiplying \eqref{l33} with
$e^{2\beta t}$ (with $\beta\leq\beta_0$) and adding up the second inequality of \eqref{est:w}, 
we arrive at
%
\begin{multline}\label{l33b}
\frac{d}{dt} \|e^{\beta t}\sqrt{{\vr_0^1}} z(t)\|_2^2 + (\|e^{\beta t}\nabla z\|_2^2+\|e^{\beta t}\nabla w\|_2^2)  \leq 2\beta e^{2\beta t}\|\sqrt{{\vr_0^1}} z(t)\|_2^2  \\[5pt]
\left(\|e^{\beta t}t^{1/2}\nabla u^2\|_\infty^2+\|e^{\beta t}t^{1/2}\nabla Q^2\|_4^2+f_1(t)\right) \int_0^T (\|\nabla z\|_2^2+\|\nabla w\|_2^2) \dt 
+ \|e^{\beta t}\sqrt{\vr_0^2}u_t^2\|_2^2 \|\dvr_0\|_X^2 \\[5pt]
+\|e^{\beta t}w_t\|_{3/2}\|e^{\beta t}\sqrt{\vr_0^1}z\|_2^{1/2}\|e^{\beta t}z\|_{H^1}^{1/2}. 
\end{multline}
For small $\beta$ one can absorb the first term on the right-hand side due to Poincar\'e inequality. By \eqref{eq:regularity}, provided $\beta \leq \beta_0$ we have 
\begin{equation} \label{l33c1}
\|e^{\beta t}t^{1/2}\nabla u^2\|_\infty^2+\|e^{\beta t}t^{1/2}\nabla Q^2\|_4^2+\|e^{\beta t}\sqrt{\vr_0^2}u_t^2\|_2^2 \leq f_1(t),     
\end{equation}
while, by Lemma \ref{Lem:9},
$$\displaylines{
\|e^{\beta t}w_t\|_{3/2}\|e^{\beta t}\sqrt{\vr_0^1}z\|_2^{1/2}\|e^{\beta t}z\|_{1,2}^{1/2} \hfill\cr\hfill\leq 
\biggl(f_{4/3}(t) \biggl(\int_0^t \|\nabla \du\|_2^2\dtau\biggr)^{1/2} +
 f_4(t)\|\nabla \du\|_2 \biggr)\|e^{\beta t}\sqrt{\vr_0^1}z\|_2^{1/2}\|e^{\beta t}z\|_{1,2}^{1/2}.}$$

Taking advantage of Poincar\'e inequality and Lemma \ref{Lem:9}, we bound the right-hand side of the above inequality as follows:
$$\displaylines{
f_{4/3}(t) \biggl(\int_0^t\! \|\nabla \du\|_2^2 \dtau\biggr)^{1/2} \|e^{\beta t}\sqrt{\vr_0^1} z\|_2^{1/2}\|e^{\beta t}\nabla z\|^{1/2}_2 \hfill\cr\hfill\leq 
 f_{4/3}^{4/3}(t) \int_0^t\! \|\nabla \du\|_2^2 \dtau+ f_{4/3}^{2/3}
 \|e^{\beta t}\sqrt{\vr_0^1} z\|_2\|e^{\beta t}\nabla z\|_2,}$$
 whence
\begin{multline} \label{l33c}
f_{4/3}(t) \biggl(\int_0^t \|\nabla \du\|_2^2 \dtau\biggr)^{1/2} \|e^{\beta t}\sqrt{\vr_0^1} z\|_2^{1/2}\|e^{\beta t}\nabla z\|^{1/2}_2 \\\leq f_1(t)\int_0^t \|\nabla \du\|_2^2\dtau+ \ep  \|e^{\beta t}\nabla z\|_2^2 + f_1(t) \|e^{\beta t}\sqrt{\vr_0^1} z\|_2^2,
\end{multline}
and
\begin{equation} \label{l33d}
\begin{aligned}
    f_4(t)\|\nabla \du\|_2  \|e^{\beta t}\sqrt{\vr_0^1} z\|_2^{1/2}\|e^{\beta t}\nabla z\|^{1/2}_2&\leq 
    \ep \|\nabla \du\|_2^2 +
    f_4^2(t) \|e^{\beta t}\sqrt{\vr_0^1} z\|_2\|e^{\beta t}\nabla z\|_2\\ &\leq 
    f_1(t)\|e^{\beta t}\sqrt{\vr_0^1} z\|_2^2+
    \ep (\|e^{\beta t}\nabla z\|_2^2 + \|\nabla \du\|_2^2).
\end{aligned}
\end{equation}
The $\ep$ terms from \eqref{l33c} and \eqref{l33d} can again be absorbed by the left-hand side of \eqref{l33b}.
Then, plugging \eqref{l33c1}-\eqref{l33d} in \eqref{l33b} we obtain 
\begin{multline} \label{l33e}
\frac{d}{dt} (e^{2\beta t}\|\sqrt{{\vr_0}} z(t)\|_2^2) +  e^{2\beta t}(\|\nabla z\|_2^2
+\|\nabla w\|_2^2) \\ \leq
f_1(t) \left(e^{2\beta t}\|\sqrt{{\vr_0}} z(t)\|_2^2) +  \int_0^t e^{2\beta s}(\|\nabla z\|_2^2
+\|\nabla w\|_2^2)\ds \right) + f_1(t) \|\dvr_0\|_X^2.
\end{multline}
Denoting  
\begin{equation*}
    G(t)=e^{2\beta t}\|\sqrt{{\vr_0}} z(t)\|_2^2 + \int_0^t e^{2\beta s}(\|\nabla z\|_2^2
+\|\nabla w\|_2^2) \ds,
\end{equation*}
we rewrite \eqref{l33e} as
\begin{equation} \label{eq:G}
    \frac{d}{dt}G(t) \leq f_1(t)G(t) + f_1(t)\|\dvr_0\|_X.
\end{equation}
We have $G(0)=\|\sqrt{\vr_0}\dv_0\|_2^2$, therefore \eqref{eq:G} implies 
\begin{equation} \label{decay:G}
    G(t) \leq \|\sqrt{\vr_0} \dv_0\|_2^2 e^{\int_0^t f_1(\tau) \dtau} + 
    \|\dvr_0\|_X \int_0^t e^{\int_s^t f_1(\tau) d\tau} f(s)\ds.
\end{equation}
Notice that the first inequality of \eqref{est:w} implies that
\begin{equation} \label{decay:w}
\sup_{t\in \R_+} e^{2\beta t}\|w(t)\|_2^2 \leq f_1(t)\int_0^t \|\nabla \du(\tau)\|_2^2\dtau \leq f_1(t)G(t).   
\end{equation}
Combining \eqref{decay:G} and \eqref{decay:w}, we obtain the following result:
\begin{lem}
Assume that \eqref{x3} holds. Then, there exist a positive 
constant $\beta<\beta_0$ and $C_0$ depending only on the data
such that
\begin{equation} \label{decay:deltau}
    \sup_{t\in [0,\infty)} 
    e^{2\beta t} \|\sqrt{\vr_0} \du \|_2^2 +\int_0^\infty 
    e^{2\beta t} \|\nabla \du\|_2^2 \dt \leq C_0 \bigl(\|\sqrt{\vr_0} \dv_0\|_2^2 + C_2 \|\dvr_0\|_X^2\bigr)\cdotp
\end{equation}
\end{lem}


\subsection{The general case}

The estimate \eqref{decay:deltau} has been obtained
under assumption  \eqref{def:xnorm}. It may happen however that the denominator of \eqref{def:xnorm}
is zero on a subset 
with positive measure, while the numerator is not. To overcome this obstacle, instead of comparing solutions 
emanating from $(\vr_0^1,v_0^1)$ and $(\vr_0^2,v_0^2)$ directly we compare each of them to an appropriate intermediate solution satisfying \eqref{def:xnorm}.   


\smallbreak
To proceed, let us denote by $(\rho^{3/2},v^{3/2},P^{3/2})$ the  velocity solution to \eqref{sys} emanating from $\left( \frac12 (\vr_0^1+\vr_0^2),v_0^2 \right)$ and by $(\eta^{3/2},u^{3/2},Q^{3/2})$ the corresponding solution in Lagrangian coordinates. Then, we look at the following  differences between solutions (recall that the density component of the solution in Lagrangian coordinates is constant in time): 
\begin{equation*}
    \begin{aligned}
    &(\dvr^I,\du^I)=(\frac12 (\vr_0^1+\vr_0^2)-\vr_0^1,\, u^{3/2}-u^1 )\\
    \andf&(\dvr^{II},\du^{II})=(\frac12 (\vr_0^1+\vr_0^2)- \vr_0^2,\, u^{3/2}-u^2).
    \end{aligned}
\end{equation*}
Clearly,  we have 
\begin{equation} \label{delta}
\dvr^I=-\dvr^{II}=\frac12 (\vr_0^2-\vr_0^1)\andf \du=\du^I-\du^{II},
\end{equation}
and  the condition \eqref{x3} is fulfilled in both cases, i.e.
\begin{equation*}
    \supp \dvr_0^{I,II} \subset \supp \frac12 (\vr_0^1 + \vr_0^2).
\end{equation*}
Let us define 
$$
\|\dvr_0^I\|_X = \Big\| \frac{\dvr_0^I}{\sqrt{\vr_0^1+\vr_0^2}}\Big\|_4=\Big\|\frac{\dvr_0^{II}}{\sqrt{\vr_0^1+\vr_0^2}}\Big\|_4\cdotp
$$
Note that, obviously
\begin{equation*}
    |\frac{\vr_0^1-\vr_0^2}{\sqrt{\vr_0^1+\vr_0^2}}|^2 =
    |\vr_0^1-\vr_0^2| |\frac{\vr_0^1-\vr_0^2}{\vr_0^1+\vr_0^2}|
    \leq |\vr_0^1-\vr_0^2|,
\end{equation*}
whence
\begin{equation} \label{dvr}
\|\dvr_0^I\|_X \leq \|\vr_0^1-\vr_0^2\|_2^{1/2}.     
\end{equation}
Denoting $\dQ^I=Q^{3/2}-Q^1$,  we obtain the following system for $(\du^I,\dQ^I)$:  
\begin{equation}\label{NSL-uniq:2}
\left\{ \begin{aligned}
  &{\vr_0^1}\du^I_t - \Delta_{u^1}\du^I + \nabla_{u^1}\dQ^I= 
(\Delta_{u^{3/2}} -\Delta_{u^1}) u^{3/2} - (\nabla_{u^{3/2}}-\nabla_{u^1}) Q^{3/2}-\dvr^I \,u^{3/2}_t,\\
&\div_{\!u^1}\du^I=(\div_{\!u^1}-\div_{\!u^{3/2}})u^{3/2},\\
&\du^I|_{t=0}=\dv_0.
\end{aligned}\right.
\end{equation}
Setting $\dQ^{II}=Q^{3/2}-Q^2,$ we see that $(\du^{II},\dQ^{II})$ satisfies an analogous system with $(\vr^1_0,u^1,\dvr^I)$ replaced by $(\vr^2_0,u^2,\dvr^{II})$
and with the initial condition $\du^{II}|_{t=0}=0$.  
\smallbreak
Let us consider the decompositions 
$$
\du^I = z^I+w^I\andf\du^{II} = z^{II}+w^{II}, 
$$
where the components are defined by \eqref{l6} and \eqref{l21} with obvious replacements of $u^1,u^2$. 
Then, one can repeat the estimates from the previous subsection. In the end, defining 
\begin{equation} \label{def:G12}
\begin{aligned}
    &G^I(t)=e^{2\beta t}\|\sqrt{{\vr_0^1}} z^I(t)\|_2^2 + \int_0^t e^{2\beta s} (\|\nabla z^I(s)\|_2^2
+\|\nabla w^I(s)\|_2^2)\ds,\\
    &G^{II}(t)=e^{2\beta t}\|\sqrt{{\vr_0^2}} z^{II}(t)\|_2^2 + \int_0^t e^{2\beta s}(\|\nabla z^{II}(s)\|_2^2
+\|\nabla w^{II}(s)\|_2^2)\ds,
\end{aligned}
\end{equation}
one obtains the following analogs of \eqref{decay:G} (recall that $\dvr_0^I=-\dvr_0^{II}$ and that the velocity component of the initial data for $\dv^{II}$ is zero) for some  $f^I,f^{II} \in L^1(\R_+)$: 
$$
\begin{aligned}
     G^I(t) &\leq \|\sqrt{\vr_0^1} \dv_0\|_2^2 e^{\int_0^tf^I(\tau) \dtau} + 
    \|\dvr_0^I\|_X^2 \int_0^t e^{\int_s^t f^I(\tau) d\tau} f^I(s)  \ds,\\
     G^{II}(t) &\leq 
    \|\dvr_0^{I}\|_X^2 \int_0^t e^{\int_s^t f^{II}(\tau)\, d\tau} f^{II}(s) \ds,
\end{aligned}    $$
Summing  up the above inequalities and using \eqref{delta} and \eqref{dvr} we obtain 
\begin{equation*} 
\sup\limits_{t\in\R_+} \left\{ e^{\beta t}\left[\|\sqrt{\vr^1_0}z^I(t)\|_2+\|\sqrt{\vr^2_0}z^{II}(t)\|_2 \right] \right\} 
+ \|e^{\beta t}\nabla \du\|_{L^2(\R_+,L^2)}\leq C \left( \|\sqrt{\vr^1_0}\dv_0\|_2 + \|\dvr_0\|_2^{1/2} \right)\cdotp
\end{equation*}
Combining this estimate with analogs of \eqref{decay:w}, where $(w,\du,G)$ has been replaced by 
$(w^{I},\du^{I},G^{I})$ and $(w^{II},\du^{II},G^{II}),$ we obtain \eqref{decay:final}, which completes the proof of Theorem \ref{thm:lag}.


\subsection{Back to the Euler perspective}

In this part, we want to translate the stability result obtained in Theorem \ref{thm:lag}  in terms of Eulerian coordinates, proving Corollary \ref{cor:eul}. 
As already pointed out in the introduction, in the case of only 
bounded initial densities, 
getting a relevant  information 
on $\vr^1(t,\cdot)-\vr^2(t,\cdot)$ in some $L^p$ space (even for $p=1$)
is hopeless; it is more reasonable to compare  functions along two different characteristics fields. 
Here we want to adopt  the language of kinetic theory, 
 considering quantities along characteristics/trajectories,
defined in  terms of Wasserstein metrics like in e.g. \cite{PR,MPesz}.
\smallbreak
More precisely, denote by $X^1$ and $X^2$ the flows defined by \eqref{L-E} for, respectively, $v^1$ and $v^2,$ and take a smooth function $\phi:\Omega \to \R.$ 
Then we consider, for each $t\geq0,$ the quantity
$$I_\phi(t):= \int_\Omega \Big(\vr^1(t,x) -  \vr^2(t,x)\Big)\phi (x) \dx.$$
%
Then, using the change of variables (of Jacobian $1$) $x=X^1(t,y)$ and $x=X^2(t,y)$
for $\vr^1$ and $\vr^2,$ respectively, yields 

\begin{align} \label{eq:deltarho}
I_\phi(t)&=
    \int_\Omega \Big[\vr_0^1 (y) \phi(X^1(t,y)) -  \vr_0^2(y) \phi (X^2(t,y)) \Big]\dy\nonumber\\
    &=\underbrace{\int_\Omega (\vr_0^1(y)- \vr_0^2(y)) \phi(X^1(t,y)) \dy}_{A_1(t)} + 
    \underbrace{\int_\Omega  \vr_0^2(y) (\phi(X^1(t,y)) -\phi(X^2(t,y)))\dy}_{A_2(t)}.
\end{align}
In order to find out the right level of regularity for $\phi,$ let us first
examine the term $A_2(t).$
For suitable functions $\underline{\eta},$ we set
$$\eta^1(t,x):=\underline{\eta} (t,Y^1(t,x)) \andf \eta^2(t,x):=\underline{\eta} (t,Y^2(t,x))\with Y^i(t,\cdot):=(X^i(t,\cdot))^{-1}.$$
Then, using again the fact that $X^1$ and $X^2$ are measure preserving,
we have
\begin{equation} \label{eq:neg}
\begin{aligned}
    &\|\eta^1(t) - \eta^2(t)\|_{W^{-1}_p} =\sup \left\{ 
    \int_\Omega (\eta^1(t,x)-\eta^2(t,x))\phi(t,x)\dx: \phi \in W^1_{p'}(\Omega), \;
    \|\phi\|_{W^1_{p'}} \leq 1 
    \right\}\\[5pt] 
    &= \sup 
    \left\{ \int_\Omega \underline{\eta}(y) [\phi(X^1(t,y))-\phi(X^2(t,y))]\dy : \phi \in W^1_{p'}(\Omega), \;
    \|\phi\|_{W^1_{p'}} \leq 1 \right\}\cdotp
\end{aligned}
\end{equation}
%
Following \cite{Maja}, in order to handle the term $\phi(X^1(t,y)) -\phi(X^2(t,y)),$ we consider 
the  family of  intermediate measure preserving flows  $(X^s)_{1\leq s\leq2}$
between $X^1$ and $X^2$ defined as follows:
\begin{equation*}
    \frac{dX^s}{dt}(t,y)= (2-s) v^1(t,X^s(t,y)) + (s-1) v^2 (t,X^s(t,y)), \qquad 
    X^s(0,y) =y.
\end{equation*}
By the chain rule and the definition of $X^s,$ we have  
$$\begin{aligned}
  \phi(X^2(t,y)) -\phi(X^1(t,y))&= \int_1^2 \frac{d}{ds} \phi(X^s(t,y)) \ds\nonumber\\  &=
  \int_1^2 \biggl(\frac d{ds} X^s(t,y)\biggr)\cdotp \nabla \phi (X^s(t,y)) \ds.
\end{aligned}$$
{}From the definition of $X^s,$ we discover that
$$\frac d{ds} X^s(t,y)=\int_0^t\bigl(v^2(t',X^s(t',y))-v^1(t',X^s(t',y))\bigr)\dt',$$
which implies that
\begin{equation}\label{eq:Xs}
\biggl\|\frac d{ds} X^s(t,\cdot)\biggr\|_p\leq \int_0^t\|\dv\|_p\dt'.
\end{equation}
Furthermore, as said before,  $X^s(t,\cdot)$ is 
measure preserving  for all $t\geq 0,$ and thus 
    $\| \nabla \phi (X^s(t,\cdot)) \|_p = \|\nabla \phi\|_p.$
In the end, using the embedding $H^1_0\hookrightarrow L_p$ for all $1\leq p<\infty$ if
$d=2,$ and all $1\leq p\leq 6$ if $d=3,$ we obtain for these values of $p$ and all $t\geq0,$
\begin{equation} \label{est:A2}
    A_2(t)\leq 
    C\|\vr_0^2\|_\infty  \left( \int_0^\infty e^{2\beta t} \| \nabla \dv\|_2^2 \,dt\right)^{1/2}\|\phi\|_{1,p'}.
\end{equation}
The above inequality reveals that one can take  the functions $\phi$ in the space
\begin{equation*}
    \phi \in W^{1}_{1^+}(\Omega) \mbox{ \ \ if \ \ } d=2
\andf
    \phi \in W^1_{6/5}(\Omega) \mbox{ \ \ if \ \ } d=3.
\end{equation*}
Bounding the term $A_1(t)$ is simpler.
Under the above assumptions on $p,$ we have $W^1_{p'}\hookrightarrow L_2.$
Hence,  using again that $X^1$ is measure preserving, we may write by Cauchy-Schwarz inequality, 
\begin{equation} \label{est:A1}
    A_1(t) 
    \leq  \|\dvr_0\|_2 \|\phi \|_2 \leq C \|\dvr_0\|_2 \|\phi \|_{1,p'}. 
\end{equation}
Altogether, by \eqref{eq:deltarho}, \eqref{eq:neg}, \eqref{est:A2}, \eqref{est:A1} and \eqref{decay:final},  we get
for any finite $p>1$ if $d=2$ and for $p\leq6$ if $d=3,$
\begin{equation} \label{est:deltarho:final}
   \sup_{t\in\R_+} \|\vr^1(t,\cdot) - \vr^2(t,\cdot)\|_{W^{-1}_p} \leq C\biggl(
 \|\vr_0^2\|_\infty  \left( \int_0^\infty e^{2\beta t} \| \nabla \dv\|_2^2 \,dt\right)^{1/2} +\|\dvr_0\|_2\biggr)\cdotp
\end{equation}
Next, let us turn to the velocity estimate. We start from the relation
$$
\du (t,y) = v^2(t,X^2(t,y))-v^1(t,X^1(t,y)),
$$
so
$$    \nabla_y \du (t,y)=  \nabla_yX^2(t,y) \cdotp\nabla_x v^2(t,X^2(t,y))
    - \nabla_yX^1(t,y)\cdotp\nabla_x v^1(t,X^1(t,y)).$$
Hence, denoting $\dX:=X^2-X^1,$ we may write
$$\begin{aligned}
    \nabla_y \du (t,y)
&=\nabla_y\dX(t,y)\cdotp\nabla v^2(t,X^2(t,y))+
\nabla_yX^1(t,y)\cdotp\nabla \dv (t,X^2(t,y))\\&\hspace{5cm}+
\nabla_yX^1(t,y)\cdotp\left[\nabla v^1 (t,X^2(t,y))-\nabla v^1(t,X^1(t,y)) \right]\\&=:K_1+K_2+K_3.
\end{aligned}$$
First, we observe that 
$$\nabla\dv(t,X^2(t,y))={}^T\!A_u^1 K_2.$$
Hence, taking the $L_2(\Omega)$ norm, and using that $X^2(t,\cdot)$ is measure preserving, we get
\begin{equation}\label{eq:dv}
 \|\nabla\dv\|_{2}\leq \|A_u^1\|_\infty  \|K_2\|_{2}\leq \|A_u^1\|_\infty\bigl(\|\nabla\du\|_2
 +\|K_1\|_2+\|K_3\|_2\bigr)\cdotp
\end{equation}
Next, from the definition of $X^1$ and $X^2$ in terms of the Lagrangian velocity, we have
\begin{equation*}
    \|K_1\|_2\leq Ct^{1/2} \left(\int_0^t \|\nabla\du\|_2^2\,dt'\right)^{1/2} \|\nabla v^2\|_\infty.
\end{equation*}
For bounding $K_3,$ we first notice that by the mean value theorem and the definition 
of the intermediate flow $X^s$, we have 
$$K_3=\nabla_yX^1(t,y)\cdotp\biggl(\int_1^2\nabla^2 v^1(t,X^s(t,y))\ds\biggr)\cdotp
\biggl(\frac d{ds}X^s(t,y)\biggr)\cdotp$$
Hence, since $X^s$ is measure preserving, using H\"older inequality and \eqref{eq:Xs} allows to get
$$\begin{aligned}
    \|K_3\|_2 &\leq \|\nabla_yX^1(t)\|_\infty \int_1^2\bigl\|\nabla^2 v^1(t,X^s(t))\bigr\|_3
    \Bigl\|\frac d{ds}X^s(t)\Bigr\|_6\ds\\
&    \leq C\|\nabla_yX^1(t)\|_\infty 
\|\nabla^2 v^1(t)\|_3 \int_0^t \|\dv(t')\|_6\,dt'\\
    &\leq C\|\nabla_yX^1(t)\|_\infty 
    \|t^{1/2}\nabla^2v^1(t) \|_3 \left(\int_0^t \|\nabla \dv\|_2^2\,dt'\right)^{1/2}.
\end{aligned}$$
Altogether, remembering \eqref{eq:dv}, we obtain for all $t\geq0,$
$$
\|\nabla\dv(t)\|_2^2 \leq C \left[ \|\nabla \du(t)\|_2^2 +\|t^{1/2}\nabla v^2\|_\infty^2\int_0^t \|\nabla \du\|_2^2\dt' 
+\|t^{1/2}\nabla^2 v^1\|_3^2\int_0^t\|\nabla \dv\|_2^2\dt'\right]\cdotp 
$$
First multiplying both sides by $e^{2\beta t}$, next integrating in time and 
using Theorem \ref{thm:lag} combined with the properties \eqref{eq:regularity}, and finally applying
Gronwall lemma, we obtain for all $t\geq0,$
\begin{equation*}
    \int_0^t e^{2\beta t'} \|\nabla\dv\|_2^2\dt'
    \leq C_0 \left( \|\sqrt{\vr^1_0}\dv_0\|_2^2 + \|\dvr_0\|_2 \right)
    \exp\left(C\int_0^t e^{2\beta t'}\|t'^{1/2}\nabla^2v^1\|_3^2\dt'\right)\cdotp
\end{equation*}
The term in the exponential may be bounded thanks to \eqref{eq:regularity:4}.
Hence, reverting to  \eqref{est:deltarho:final}
completes the proof of Corollary \ref{cor:eul}.



%
%
%
%


\section{Decay estimates} \label{s:decay}

The goal of this section is to show that the solutions provided by Theorems \ref{thm:d=2} and \ref{thm:d=3}
indeed satisfy Properties \eqref{eq:regularity:8} to \eqref{eq:regularity:7}. This will be an immediate 
consequence of Lemmas \ref{l:decay1}, \ref{l:decay2} and \ref{l:decay3} below. 
\smallbreak
Before tackling the proof, we need to recall elementary properties
of the  solutions to \eqref{sys}. The first one is the  conservation of any $L_p$ norm of the density, 
a consequence of the divergence free property of the velocity field : we have 
$\|\vr(t)\|_p = \|\vr_0\|_p$.    

Of course, as $v$ in $H^1_0(\Omega),$ we have the Poincar\'e inequality: 
\begin{equation} \label{Poin}
\|v\|_{2} \leq C_P \|\nabla v\|_2.    
\end{equation}

\subsection{Decay of space derivatives}
The starting point is the following decay estimate which is a direct consequence of
diffusion and  Poincar\'e inequality:
\begin{lem}
Let $(\vr,v)$ be a solution to \eqref{sys} given either by Theorem \ref{thm:d=2} or by Theorem \ref{thm:d=3}. Then 
\begin{equation} \label{decay:1}
\forall t\geq0,\; 
\int_{\Omega} \vr(t)|v(t)|^2\dx \leq e^{-\beta_1 t}\int_{\Omega} \vr_0|v_0|^2\dx,    
\mbox{ \ \ where  \ \ } \beta_1 =\frac{2\mu}{\vr^* C_P^2}\cdotp 
\end{equation}
\end{lem}
\begin{proof} The proof is independent of the space dimension. We start with the classical energy estimate that is obtained testing the momentum equation by $v,$ namely
\begin{equation} \label{ene:1}
\frac12\frac{d}{d t}\int_{\Omega} \vr|v|^2\dx + \mu\int_{\Omega}|\nabla v|^2\dx = 0. 
\end{equation}
Hence, remembering  \eqref{Poin}, we get 
$$\frac{C_P^2}{2\mu} \frac{d}{d t}\int_{\Omega}\vr|v|^2\dx+\int_{\Omega}|v|^2\dx \leq 0.$$
Multiplying both sides by $\vr^*$ and using  the obvious fact that
$\vr^*|v|^2 \geq \vr |v|^2,$
we obtain 
$$
\frac{\vr^*C_P^2}{2\mu} \frac{d}{d t}\int_{\Omega}\vr|v|^2\dx+\int_{\Omega}\vr|v|^2\dx \leq 0,
$$
from which we conclude to \eqref{decay:1}. 
\end{proof}

Our next aim is to establish a  decay estimate for the gradient of the velocity.
\begin{lem} \label{l:decay1}
Let $(\vr,v)$ be a solution to \eqref{sys} given either by Theorem \ref{thm:d=2} or by Theorem \ref{thm:d=3}. Then, there exist  positive constants  $C_0$ and $C_{0,p}$ depending only on the data
(and on $p$ for second one), 
and\footnote{One can take $\beta_2$ of the 
form $\beta_2=c_\Omega \mu /(\vr^*\delta^2)$ where $\delta$ stands for the 
diameter of $\Omega$ and $c$ depends only on the shape of $\Omega.$}
$\beta_2<\beta_1$  such that for all $t\geq0,$ we have 
\begin{align} 
&\|\nabla v(t)\|_2 \leq C_0e^{-\beta_2 t}, \label{decay:2}\\[3pt]
&\|v(t)\|_p \leq C_{0,p}e^{-\beta_2 t}, \label{decay:2c}\\[3pt]
&\int_0^{\infty} e^{2\beta_2 t}\left[\|\sqrt{\vr(t)} v_t(t)\|_2^2 + \|\nabla^2 v(t)\|_2^2 + \|\nabla P(t)\|_2^2\right]\dt \leq C_0,
\label{decay:2a}
\end{align}
where, in \eqref{decay:2c}, one can take any $p\in[1,\infty)$ if $d=2,$
and any $p\in[1,6]$ if  $d=3.$ 
\end{lem}

\begin{proof}
Clearly, performing a suitable time, space, density and velocity rescaling 
reduces the proof to the case $\vr^*=\mu=1,$ in a domain of diameter $1.$ 
Hence, we shall only prove the lemma in this case for notational simplicity.

Now, testing the momentum equation by $v_t$ yields 
$$\frac12\frac{d}{d t}\int_{\Omega} |\nabla v|^2\dx
+\int_{\Omega}\vr |v_t|^2 \dx =-\int_\Omega \vr v_t\cdot(v\cdot\nabla v)\dx\leq \frac12\int_{\Omega}\vr |v_t|^2 \dx
+\frac12\int_{\Omega}\vr|v\cdot\nabla v|^2\dx.
$$
The classical theory for the Stokes system yields for some $C_\Omega$ depending only 
on the shape of $\Omega,$
$$\|\nabla^2 v\|_2^2+\|\nabla P\|_2^2\leq C_\Omega \bigl(\|\vr v\cdot\nabla v\|_2^2+ \|\vr v_t\|_2^2\bigr)\cdotp$$
Hence we have (remember that $\vr^*=1$): 
\begin{equation} \label{ene:2}
\frac{d}{d t}\int_{\Omega} |\nabla v|^2\dx+\frac12\int_{\Omega}\bigl(\vr|v_t|^2 +
\frac1{C_\Omega}(|\nabla^2 v|^2+|\nabla P|^2)\bigr)\dx 
\leq \frac32\int_{\Omega}\vr|v\cdot\nabla v|^2\dx.
\end{equation}
Adding this inequality to \eqref{ene:1}, we get (up to a change of $C_\Omega$)
\begin{multline} \label{2.2}
\frac{d}{d t}\int_{\Omega}[|\nabla v|^2+\frac{1}{2}\vr|v|^2]\dx +\int_{\Omega}[|\nabla v|^2+\frac{1}{2}\vr|v|^2]\dx 
+\frac12\int_{\Omega}\bigl(\vr|v_t|^2 +\frac1{C_\Omega}(|\nabla^2 v|^2+|\nabla P|^2)\bigr)\dx\\ 
\leq \frac32\int_{\Omega}\vr|v\cdot\nabla v|^2\dx + \frac{1}{2}\int_{\Omega}\vr|v|^2\dx. 
\end{multline}
In order to estimate the first term on the right in the two dimensional case we first write 
$$\int_{\Omega}\vr|v\cdot\nabla v|^2\dx 
\leq \|\sqrt{\vr}v\|_2\|v\|_\infty\|\nabla v\|_4^2.$$
Hence, using the following two interpolation inequalities:
\begin{equation} \label{Lad}
\|z\|_4 \leq C \|z\|_2^{1/2}\|\nabla z\|_2^{1/2}\ \hbox{ and }\ 
\|z\|_\infty\leq C\|z\|_2^{1/2}\|\nabla^2z\|_2^{1/2},
\end{equation}
and remembering that $\vr^*=1,$ we get
$$\begin{aligned} 
\frac32\int_{\Omega}\vr|v\cdot\nabla v|^2\dx &\leq C \|\sqrt{\vr}v\|_2\|v\|_2^{1/2}\|\nabla v\|_2\|\nabla^2v\|_2^{3/2}\\
&\leq \frac1{4C_\Omega}\|\nabla^2v\|_2^2 + C
\|\sqrt{\vr}v\|_2^4\|v\|_2^{2}\|\nabla v\|_2^4.
\end{aligned}$$
The first term can be absorbed by the left-hand side of \eqref{2.2} and 
one can bound $\|v\|_2$ from \eqref{Poin}. Hence,  we get
for some $\gamma>0$ depending only on the shape of the domain, 
\begin{multline} \label{2.6}
\frac{d}{d t}\int_{\Omega}[|\nabla v|^2+\frac{1}{2}\vr|v|^2]\dx +\int_{\Omega}[|\nabla v|^2+\frac{1}{2}\vr|v|^2]\dx 
+\frac12\int_{\Omega}\bigl(\vr|v_t|^2 +\frac1{2C_\Omega}(|\nabla^2 v|^2+|\nabla P|^2)\bigr)\dx\\
\leq Ce^{-\gamma t} [1 + \|\nabla v\|_2^6].\end{multline}
The crucial observation is that $\|\nabla v(t)\|_2$ may  bounded uniformly in time in terms
of the data (see Propositions 3.1 and 3.3 in \cite{DM1}). Therefore, we obtain
\begin{equation} \label{2.6ba}
\frac{d}{d t}\int_{\Omega}[|\nabla v|^2+\frac{1}{2}\vr|v|^2]\dx 
+\frac12 \int_{\Omega}\bigl(|\nabla v|^2+\frac{1}{2}\vr|v|^2+\vr|v_t|^2 
+\frac1{2C_\Omega}(|\nabla^2 v|^2+|\nabla P|^2)\bigr)\dx
\leq C_0 e^{-\gamma t}.  
\end{equation}
This leads for any $\gamma_1 \in\R$ to 
\begin{multline} \label{2.6b}
\frac{d}{d t} \left[ e^{\gamma_1 t}\int_{\Omega}\left(|\nabla v|^2+\frac{1}{2}\vr|v|^2\right)\dx\right]\\
 +\Bigl(\frac12-\gamma_1\Bigr) \int_{\Omega}e^{{\gamma_1} t}\bigl(|\nabla v|^2+\frac{1}{2}\vr|v|^2+\vr|v_t|^2 
+\frac1{2C_\Omega}(|\nabla^2 v|^2+|\nabla P|^2)\bigr)\dx
\leq C_0 e^{-(\gamma -\gamma_1)t}.  
\end{multline}

In the three dimensional case, the only difference is the slightly more complicated treatment of the right-hand side of \eqref{2.2}. Nevertheless, by using H\"older inequality and Gagliardo-Nirenberg inequality, 
we arrive (still using that $\vr^*=1$) at:
$$\begin{aligned}
\int_{\Omega}\vr |v\cdot\nabla v|^2 &\leq 
\|\vr^{1/4} v\|_4^2 \|\nabla v\|_4^2\\
&\leq 
C\|\sqrt{\vr}v\|_2^{1/2}\|v\|_6^{3/2}\|\nabla v\|_2^{1/2}\|\nabla^2 v\|_2^{3/2}\\
&\leq C\|\sqrt{\vr}v\|_2^{1/2}\|\nabla v\|_2^2\|\nabla^2 v\|_2^{3/2} \\
&\leq \frac{1}{2C_\Omega}\|\nabla^2 v\|_2^2 + C\|\sqrt{\vr}v\|_2^2\|\nabla v\|_2^8 \leq 
\frac{1}{2C_\Omega}\|\nabla^2 v\|_2^2 + C_0e^{-\beta_1 t}, \end{aligned} $$
where in the last passage we have used \eqref{decay:1} and uniform boundedness of $\|\nabla v(t)\|_2$.
Again, the first term can be absorbed by the left hand side of \eqref{2.2} and we obtain \eqref{2.6b}. 

Now, choosing e.g. $\gamma_1=\min(\gamma,1/4)$ and integrating \eqref{2.6b} on $[0,t]$ yields
$$\displaylines{e^{\gamma_1t}\Bigl(\|\nabla v(t)\|_2^2+\frac12\|\sqrt{\vr(t)}\,v(t)\|_2^2\Bigr)
+\Bigl(\frac12-\gamma_1\Bigr)\int_0^te^{\gamma_1s}\Bigl(\|\nabla v\|_2^2+\frac12\|\sqrt\vr\,v\|_2^2
\hfill\cr\hfill+\frac12\|\sqrt\vr\,v_t\|_2^2+\frac1{2C_\Omega}(\|\nabla^2 v\|_2^2+\|\nabla P\|_2^2)\Bigr)\ds
\leq\frac{C_0}{\gamma-\gamma_1}
+\|\nabla v_0\|_2^2+\frac12\|\sqrt{\vr_0}\,v_0\|_2^2,}$$
which readily gives \eqref{decay:2} and \eqref{decay:2a}.
As for \eqref{decay:2c}, it just
results from Poincar\'e inequality and Sobolev embedding.

\end{proof}


\subsection{Decay of time derivatives}
Here we  estimate  time and time-space derivatives.
\begin{lem}\label{l:decay2}
Let $(\vr,v)$ be a solution to \eqref{sys} given either by Theorem \ref{thm:d=2} or by Theorem \ref{thm:d=3}. Then, there exists a positive constant $\beta_3<\beta_2$ (still of the form $\beta_3=c \mu/(\vr^*\delta^2)$) 
such that 
\begin{equation} \label{decay:3}
\sup_{t\in \R_+}\int_{\Omega}te^{2\beta_3  t}\vr|v_t|^2\dx +  \int_0^{\infty} e^{2\beta_3  t}\left[ \|\sqrt{t}\nabla v_t\|_2^2+ 
\|\sqrt{t} v_t\|_2^2
\right]\dt \leq C_0 
\end{equation}
for some positive constant $C_0$ depending only on the data.
\end{lem}
\begin{proof}  We keep the assumption $\delta=\mu=\vr^*=1.$
Now, differentiating the momentum equation in time and multiplying by $\sqrt{t}e^{\beta  t},$ we obtain  
\begin{multline}\label{eq:ppp}
\vr(\sqrt{t}e^{\beta  t}v_t)_t +\sqrt{t}e^{\beta  t}\vr_tv_t
-\frac{1}{2\sqrt{t}}e^{\beta  t}\vr v_t -
\beta  \sqrt{t}e^{\beta  t}\vr v_t+\sqrt{t}e^{\beta  t}\vr_tv\cdot\nabla v \\+\sqrt{t}e^{\beta  t}\vr v_t\cdot\nabla v +\sqrt{t}e^{\beta  t}\vr v\cdot\nabla v_t -\Delta(\sqrt{t}e^{\beta  t}v_t)+\nabla (\sqrt{t}e^{\beta  t}P_t)=0.\end{multline}
Testing \eqref{eq:ppp}  with $\sqrt{t}e^{\beta  t}v_t$ yields
\begin{equation} \label{2.7}
\frac{1}{2}\frac{d}{d t}\int_{\Omega}te^{2\beta  t}\vr |v_t|^2\dx+\int_{\Omega}te^{2\beta  t}|\nabla v_t|^2 = \sum_{i=1}^5 I_i,    
\end{equation}
where, using $\vr_t=-\div(\vr v)$ in the third term on the left-hand side of \eqref{eq:ppp}, we have 
\begin{align*}
&I_1 = \frac{1}{2}\int_{\Omega}e^{2\beta  t}\vr|v_t|^2\dx,
&I_2 = -\int_{\Omega}(\sqrt{t}e^{\beta  t}\vr_t v \cdot \nabla v) \cdot(\sqrt{t}e^{\beta  t} v_t)\dx,\\
&I_3 = -\int_{\Omega}(\sqrt{t}e^{\beta  t}\vr v_t \cdot \nabla v) \cdot(\sqrt{t}e^{\beta  t} v_t) \dx,
&I_4 = -2\int_{\Omega}(\sqrt{t}e^{\beta  t}\vr v \cdot \nabla v_t) \cdot(\sqrt{t}e^{\beta  t} v_t) \dx,\\
&I_5 = \beta  \int_\Omega te^{2\beta  t}\vr |v_t|^2. &
\end{align*}

Now, we estimate the right-hand side of \eqref{2.7}. From the continuity equation, we have 
$$I_2 = \int_{\Omega} te^{2\beta  t}\div(\vr v) (v \cdot \nabla v) \cdot v_t \dx = 
- \int_{\Omega} t  e^{2\beta  t} \vr v \cdot \nabla[(v\cdot\nabla v)\cdot v_t]\dx.$$
Therefore, 
$$|I_2| \leq \int_{\Omega}t e^{2\beta  t} \vr|v|\left( |\nabla v|^2 |v_t| + |v||\nabla^2 v||v_t| + |v||\nabla v||\nabla v_t| \right)\dx =: I_{21}+I_{22}+I_{23}.  $$
For $I_{21}$ in the two-dimensional case, we have 
$$\begin{aligned}
|I_{21}| &\leq \int_\Omega\bigl(\sqrt{t\vr}|v||\nabla v|^2\sqrt{t\vr}|v_t|e^{2\beta  t}\bigr)\dx\\
&\leq \|v\|_{\infty}^2\|\sqrt{t\vr} e^{\beta  t}v_t\|_2^2 + t e^{2\beta  t}\|\nabla v\|_4^4,\end{aligned}
$$
so using \eqref{Lad} and \eqref{decay:2} we get  
\begin{eqnarray} \label{2.27}
|I_{21}| &\!\!\!\leq\!\!\!& \|v\|_{\infty}^2\|\sqrt{t\vr}e^{\beta  t}v_t\|_2^2 + t\|\nabla v\|_2^2 e^{2\beta  t}\|\nabla^2 v\|_2^2\nonumber\\
&\!\!\!\leq\!\!\!& \|v\|_{\infty}^2\|\sqrt{t\vr}e^{\beta  t}v_t\|_2^2 + Ce^{-2\beta_2 t}
\|e^{\beta  t}\nabla^2 v\|_2^2.
\end{eqnarray}

In the three dimensional case we do not have \eqref{Lad}, but we can write using H\"older inequality
and the embedding $H^1_0\hookrightarrow L_6,$
$$\begin{aligned}
|I_{21}| & \leq \sqrt{t} e^{2\beta  t} \int_{\Omega}\sqrt{t\vr}|v_t||v||\nabla v|^2\dx\\ &\leq 
\sqrt{t} e^{2\beta  t} \|\sqrt{t\vr}v_t\|_4 \|v\|_6 \|\nabla v\|_{24/7}^2 \\
&\leq \sqrt{t} e^{\beta t} \|\sqrt{t\vr}v_t\|_2^{1/4}\|\sqrt{t}e^{\beta t}v_t\|_6^{3/4}\|v\|_6\|\nabla v\|_{24/7}^2\\ 
&\leq \frac{1}{10}\|\sqrt{t}e^{\beta  t}\nabla v_t\|_2^2 + Ct^{4/5}e^{8\beta t/5} \|\sqrt{t\vr}v_t\|_2^{2/5}\|\nabla v\|_{24/7}^{16/5}\|\nabla v\|_2^{8/5}. 
\end{aligned}$$
Then, using the Gagliardo-Nirenberg inequality
$\|\nabla v\|_{24/7}^{16/5}\leq C \|\nabla v\|_2^{6/5}\|\nabla^2 v\|_2^2$ and \eqref{decay:2},
we discover that
\begin{equation*} 
\begin{aligned}
|I_{21}| &\leq \frac{1}{10}\|\sqrt{t}e^{\beta  t}\nabla v_t\|_2^2 + t^{4/5}e^{-4\beta t/5}\|\sqrt{t\vr} e^{\beta  t}v_t\|_2^{2/5} \|\nabla v\|_2^{14/5}\|e^{\beta  t}\nabla^2 v\|_2^2\\
&\leq \frac{1}{15}\|\sqrt{t}e^{\beta  t}\nabla v_t\|_2^2 + C_0e^{-(4\beta+14\beta_2)t/5}
\|e^{\beta  t}\nabla^2 v\|_2^2\|\sqrt{t\vr}e^{\beta  t}v_t\|_2^{2/5}.
\end{aligned}
\end{equation*}
Therefore, there exits $c>0$ such that 
\begin{equation} \label{I21}
|I_{21}|\leq \frac{1}{10}\|\sqrt{t}e^{\beta  t}\nabla v_t\|_2^2 + C_0e^{-ct}\|e^{\beta  t}\nabla^2 v\|_2^2\left(1+\|\sqrt{t\vr}e^{\beta  t}v_t\|_2^2 \right). 
\end{equation}
The remaining two parts of $I_2$ are simpler: we have 
\begin{eqnarray} \label{I22}
I_{22}=\int_{\Omega}\sqrt{t\vr} e^{\beta  t} |\nabla^2 v||v|^2\sqrt{t\vr}e^{\beta  t}|v_t|
&\!\!\!\leq\!\!\!& \|\sqrt{t\vr} e^{\beta  t}\nabla^2v\|_2\|v\|_6^2\|\sqrt{t\vr} e^{\beta  t}v_t\|_6\nonumber\\
&\!\!\!\leq\!\!\!& C\|\sqrt{t} e^{\beta  t}\nabla^2v\|_2\|\nabla v\|_2^2\|\sqrt{t} e^{\beta  t}\nabla v_t\|_2\nonumber\\
&\!\!\!\leq\!\!\!&  \frac{1}{15}\|\sqrt{t}e^{\beta  t}\nabla v_t\|_2^2 + C_0e^{-4\beta_2t}
\|\sqrt{t} e^{\beta  t}\nabla^2v\|_2^2,
\end{eqnarray}
and 
\begin{eqnarray} \label{I23}
I_{23} & \!\!\!\leq\!\!\!&
\|\sqrt{t\vr} e^{\beta  t}\nabla v_t\|_2\|v\|_6^2\|\sqrt{t\vr} e^{\beta  t}\nabla v\|_6\nonumber\\
& \!\!\!\leq\!\!\!& C\|\sqrt{t} e^{\beta  t}\nabla v_t\|_2\|\nabla v\|_2^2\|\sqrt{t} e^{\beta  t}\nabla^2 v\|_2\nonumber\\& \!\!\!\leq\!\!\!&
 \frac{1}{15}\|\sqrt{t}e^{\beta  t}\nabla v_t\|_2^2 + C_0e^{-4\beta_2t}
\|\sqrt{t} e^{\beta  t}\nabla^2v\|_2^2.
\end{eqnarray}
Hence, putting  \eqref{I21}, \eqref{I22} and \eqref{I23} together, we obtain for some $c>0,$
\begin{equation} \label{I2}
I_2 \leq  \frac{1}{5} \|\sqrt{t}e^{\beta  t}\nabla v_t\|_2^2 
+ C_0e^{-ct}\|e^{\beta  t}\nabla^2 v\|_2^2\bigl(1+\|\sqrt{t\vr}e^{\beta  t}v_t\|_2^2\bigr)\cdotp
\end{equation}
Next, we estimate $I_3$ as follows:
\begin{eqnarray} \label{I3}
I_3 &\!\!\!\leq\!\!\!& \|\nabla v\|_2 \|\sqrt{t\vr}e^{\beta  t}v_t\|_4^2 \nonumber\\[1ex]
&\!\!\!\leq\!\!\!&
C \|\nabla v\|_2 \|\sqrt{t\vr}e^{\beta  t}v_t\|_2^{1/2}
\|\sqrt{t}e^{\beta  t}v_t\|_6^{3/2}\nonumber\\[1ex]    
&\!\!\!\leq\!\!\!& C \|\nabla v\|_2 \|\sqrt{t\vr}e^{\beta  t}v_t\|_2^{1/2}\|\sqrt{t}e^{\beta  t}\nabla v_t\|_2^{3/2}\nonumber\\
&\!\!\!\leq\!\!\!& \frac{1}{10} \|\sqrt{t}e^{\beta  t}\nabla v_t\|_2^2 + C \|\sqrt{t\vr}e^{\beta  t}v_t\|_2^2 \|\nabla v\|_2^4\nonumber\\
&\!\!\!\leq\!\!\!& \frac{1}{10} \|\sqrt{t}e^{\beta  t}\nabla v_t\|_2^2 + C_0e^{-4\beta_2t} \|\sqrt{t\vr}e^{\beta  t}v_t\|_2^2
\end{eqnarray}
and
\begin{equation} \label{I4}
I_4 \leq 2\int_{\Omega}\vr|v||\sqrt{t}e^{\beta  t}\nabla v_t||\sqrt{t}e^{\beta  t}v_t|\dx\leq \frac{1}{10} \|\sqrt{t}e^{\beta  t}\nabla v_t\|_2^2 + C\|v\|_{\infty}^2 \|\sqrt{t\vr}e^{\beta  t}v_t\|_2^2.   
\end{equation}
Finally, as \eqref{Poin} also applies to $v_t,$  one may write for sufficiently  small $\beta,$ 
\begin{equation}\label{I5} I_5 \leq \frac{1}{10} \|\sqrt{t}e^{\beta  t}\nabla v_t\|_2^2.
\end{equation}
Combining \eqref{2.7},\eqref{I2}, \eqref{I3}, \eqref{I4} and \eqref{I5}  we arrive at 
\begin{multline} \label{2.30}
\frac{1}{2}\frac{d}{d t}\int_{\Omega}t e^{2\beta  t}\vr |v_t|^2\dx+\frac{1}{2} \int_{\Omega}te^{2\beta  t}|\nabla v_t|^2\dx \\
\leq  C_0\bigl(e^{-ct}\|e^{\beta  t}\nabla^2 v\|_2^2+\|v\|_\infty^2\bigr)\|\sqrt{t\vr} e^{\beta  t} v_t\|_2^2 \\
+C_0\bigl(\|\sqrt\vr e^{\beta  t}v_t\|_2^2+e^{-ct}\|e^{\beta  t}\nabla^2 v\|_2^2\bigr)\cdotp
\end{multline}

By virtue of Lemma \ref{l:decay1}, the last line is integrable on $\R_+$ for any $\beta \leq \beta_2$, as well as the prefactor of the second line (observe that $H^2\hookrightarrow L_\infty$). 
Hence, Gronwall inequality ensures that
\begin{equation} \label{2.31} 
\sup_{t\in \R_+}\int_{\Omega} t e^{2\beta  t} \vr|v_t|^2\dx + \int_{\R_+} e^{2\beta  t}\|\sqrt{t}\nabla v_t\|_2^2\dt <\infty.
\end{equation}
Now, Poincar\'e inequality implies the bound for $\|e^{\beta t}\sqrt{t}v_t\|_{L_2(\R_+;L_2)}$, which completes the proof.

\end{proof}


\subsection{Shift of integrability and control of  $\|\nabla v\|_{\infty}$}

Using the decay estimates we proved so far will finally enable us to establish 
similar properties for higher order norms:
\begin{lem}\label{l:decay3}
Let $(\vr,v)$ be a solution to \eqref{sys} given either by Theorem \ref{thm:d=2} or by Theorem \ref{thm:d=3}. Then, the following properties hold true:
\begin{align}  
  \label{2.10e}
&\sqrt{t}e^{\beta_2 t}(\nabla^2 v,\nabla P) \in L_p(\R_+;L_s(\Omega)) \quad 2\leq s \leq 6\andf p=\frac{4s}{3s-6}\quad\hbox{if }\ d=3,\\[3pt]
 \label{2.10ee}
&\sqrt{t}e^{\beta_2 t}(\nabla^2 v,\nabla P) \in L_p(\R_+;L_s(\Omega)) \quad 2\leq s <\infty\andf p=\frac{2s}{s-2}\quad\hbox{if }\ d=2,\\[3pt]
\label{2.10f}
&e^{\beta_4  t} \nabla^2 v \in L_1(\R_+;L_r(\Omega)) \quad \textrm{for some} \quad r>d, 
\\[3pt]
\label{2.10}
&e^{\beta_4  t}\nabla v \in L_1(\R_+;L_\infty(\Omega)),
\end{align}
for some $0<\beta_4=\frac{c\mu}{\vr^*\delta^2}<\beta_2$, where $\beta_2$ is from Lemma \ref{l:decay1}.
\end{lem}
\begin{proof} We multiply \eqref{sys}$_1$ by $\sqrt{t}e^{\beta t}$, where $\beta=\beta_2$, and rewrite it as a Stokes system: 
\begin{align} \label{Stokes}
-\Delta \sqrt{t}e^{\beta t}v + \nabla \sqrt{t}e^{\beta t}P = -\sqrt{t}e^{\beta t} \vr v_t - \sqrt{t}e^{\beta t} \vr v \cdot \nabla v, \quad
\div \sqrt{t}e^{\beta t} v=0.
\end{align}
We start with proving \eqref{2.10e}.
By the interpolation inequality 
$$
\|f\|_q \leq \|f\|_2^{(6-q)/2q}\|f\|_6^{(3q-6)/2q}, \quad 2\leq q \leq 6,
$$
it is enough to prove 
\begin{equation} \label{2.10e1}
\sqrt{t}e^{\beta  t}(\nabla^2 v,\nabla P) \in L_\infty(\R_+;L_2(\Omega)) \cap L_2(\R_+;L_6(\Omega)).    
\end{equation}
By \eqref{decay:3} and Sobolev embedding, we have 
\begin{equation} \label{2.10b}
\sqrt{\vr t}e^{\beta  t} v_t \in L_\infty(\R_+;L_2(\Omega)) \cap L_2(\R_+;L_6(\Omega)).  
\end{equation}
Therefore, the elliptic regularity of \eqref{Stokes} implies
$$\begin{aligned}
    \|\sqrt{t} e^{\beta  t}(\nabla^2 v,\nabla P)\|_{L_\infty(\R_+;L_2)} &\leq 
    C_0\!+\! C\|\vr v \cdot \nabla \sqrt{t} e^{\beta  t}
    v\|_{L_\infty(\R_+;L_2)}\\
    \leq& C_0 \!+\!C \|e^{\beta t}v\|_{L_\infty(\R_+;L_6)}
    \|\sqrt{t} \nabla v\|_{L^\infty(\R_+;L_3)}\\
    \leq& C_0\!+\! C\sqrt t \,e^{-d \beta t/6}
    \|e^{\beta  t}v\|_{L_\infty(\R_+;L_6)}
    \|\nabla v\|_{L_\infty(\R_+;L_2)}^{1-d/6}
    \|\sqrt{t}e^{\beta  t} \nabla^2 v\|_{L_\infty(\R_+;L_2)}^{d/6}.\end{aligned}$$
So, by Young inequality and \eqref{decay:2c}, we get 
\begin{equation} \label{2.11}
     \|\sqrt{t} e^{\beta  t}(\nabla^2 v,\nabla P)\|_{L_\infty(\R_+;L_2)}\leq C_0.
\end{equation}
Similarly, starting from \eqref{Stokes} and thanks to Inequality \eqref{decay:3} and embedding
$H^1_\Omega\hookrightarrow L_6,$ we have
$$\begin{aligned}
    \|\sqrt{t}e^{\beta  t}(\nabla^2 v,\nabla P)\|_{L_2(\R_+;L_6)} &\leq C_0+C
    \|\sqrt{t}e^{\beta  t} v \cdot \nabla v\|_{L_2(\R_+;L_6)} \\
    &\leq C_0 + C \|e^{\beta  t/2 } \sqrt{t} v\|_{L_\infty(\R_+;L_\infty)}
    \|e^{\beta  t/2} \nabla v\|_{L_2(\R_+;L_6)}.
\end{aligned}$$
In  three-dimensional case, we have
$$
    \|\sqrt{t} e^{\beta  t/2} v\|_{L_\infty(\R_+;L_\infty)} \leq 
    C\|\sqrt{t} e^{\beta  t/2} \nabla v\|_{L_\infty(\R_+;L_2)}^{1/2}\|\sqrt{t} e^{\beta  t/2} \nabla v\|_{L_\infty(\R_+;L_6)}^{1/2},
$$
which implies
$$
\begin{aligned}
\|\sqrt{t}e^{\beta  t} \nabla^2 v\|_{L_2(\R_+;L_6)} &\leq
C_0 + C\|e^{\beta t/2}\sqrt{t}\nabla v\|_{L_\infty(\R_+;L_2)}^{1/2}
\|e^{\beta t/2}\nabla v\|_{L_2(\R_+;L_6)}\|\sqrt{t}e^{\beta t/2}\nabla v\|_{L_\infty(\R_+;L_6)}^{1/2}\\
& \leq C_0 + C_0\|e^{\beta t}\nabla v\|_{L_\infty(\R_+;L_2)}^{1/2}\|\sqrt{t}e^{\beta t/2}\nabla v\|_{L_\infty(\R_+;L_6)}^{1/2} \\
& \leq C_0 + C_0\|e^{\beta t}\nabla v\|_{L_\infty(\R_+;L_2)}
+\|\sqrt{t}e^{\beta t}\nabla^2 v\|_{L_\infty(\R_+;L_2)},
\end{aligned}
$$
where in the first passage we used \eqref{decay:2a} and Sobolev embedding to estimate $\|e^{\beta t/2}\nabla v\|_{L_2(\R_+;L_6)}$.
\smallbreak
Therefore, by \eqref{decay:2} and \eqref{2.11}, choosing $\beta$ small enough, we obtain 
\begin{equation*}
\|\sqrt{t}e^{\beta  t}(\nabla^2 v,\nabla P)\|_{L_2(\R_+;L_6)}\leq C_0.   
\end{equation*}
This  completes the proof of \eqref{2.10e1}, and thus also of \eqref{2.10e}.
The easier two-dimensional case (that is \eqref{2.10ee}) is left to the reader. 
\medbreak
In order to show the next part of the lemma, 
we note that for $\beta_4,\delta>0$ such that $\beta_4+\delta<\beta_2$ we have %
$$
\begin{aligned}
\int_1^{+\infty} e^{\beta_4 t}\|\nabla^2 v\|_{L_s}
&\leq \left(\int_1^{+\infty} e^{-\delta p't}\dt\right)^{1/p'}
\left(\int_1^{+\infty} e^{(\beta_4+\delta)pt}\|\nabla^2 v(t)\|_{L_s}^p\dt\right)^{1/p}\\[3pt]
&\leq C \|e^{\beta_2 t}\sqrt{t} \nabla^2 v\|_{L_p(\R_+;L_s)}.
\end{aligned}$$
For small times we can write 
$$
\int_0^1\|\nabla^2 v(t)\|_r \dt \leq 
\left(\int_0^1 t^{-\alpha p'}\right)^{1/p'}\int_0^1 t^{\alpha p}\|\nabla^2 v(t)\|_r^p \dt. 
$$
We can choose for instance $p=\frac{8}{5}$, which corresponds to $s=4$ in \eqref{2.10e}, then $\alpha=\frac{5}{16}$ so that $\alpha p=\frac{1}{2}$ and $\alpha p'<1$, so the first integral is again finite. This completes the proof of \eqref{2.10f}.
\smallbreak
As for  \eqref{2.10}, it results directly from \eqref{decay:2} and \eqref{2.10f} owing to a suitable Gagliardo-Nirenberg inequality that yields: 
\begin{equation*}
\|e^{\beta_4 t}\nabla v\|_{L_1(\R_+;L_\infty)}\leq C ( \|e^{\beta_4 t}\nabla v\|_{L_1(\R_+;L_2)}+\|e^{\beta_4 t}\nabla^2 v\|_{L_1(\R_+;L_r)}.    
\end{equation*}
The right-hand side is finite whenever $\beta_4<\beta_2$.
\end{proof}




\subsection* {Acknowledgments:}
The first two authors  have  been  partially supported by the ANR project INFAMIE (ANR-15-CE40-0011). The second and third author were partially supported by National Science Centre grant
No2018/29/B/ST1/00339 (Opus).  


\begin{thebibliography}{99}

 \bibitem{AP} H. Abidi and M. Paicu: 
 {\em Existence globale pour un fluide inhomog\`ene.} Ann. Inst. Fourier, 57 (2007), no. 3, 883–917. 

\bibitem{CJ}
T. Chang and B.J. Jin: {\em Global well-posedness of the Navier-Stokes equations of an inhomogeneous fluid in the half-space with inflow boundary condition.} J. Math. Anal. Appl. 482 (2020), no. 2, 123567, 29 pp.  

\bibitem{CZZ}
D. Chen, Z. Zhang and  W. Zhao: {\em Fujita-Kato theorem for the 3-D inhomogeneous Navier-Stokes equations.} J. Differential Equations 261 (2016), no. 1, 738–761.

\bibitem{CK}
Y. Cho and H. Kim: {\em Unique solvability for the density-dependent Navier-Stokes equations.} Nonlinear Anal. 59 (2004), no. 4, 465–489. doi:10.1016/j.na.2004.07.020

\bibitem{CHW}
W. Craig, X. Huang and Y. Wang: {\em Global wellposedness for the 3D inhomogeneous incompressible Navier-Stokes equations.} J. Math. Fluid Mech. 15 (2013), no. 4, 747–758.
doi:10.1007/s00021-013-0133-6

\bibitem{D0} R. Danchin: {\em Density-dependent incompressible viscous fluids in critical spaces.} Proc. Roy. Soc. Edinburgh Sect. A 133 (2003), no. 6, 1311–1334.


\bibitem{D2}
R. Danchin: {\em Density-dependent incompressible fluids in bounded domains.} J. Math. Fluid
Mech. 8 (2006), no. 3, 333–381. doi:10.1007/s00021-004-0147-1




\bibitem{DM0} R. Danchin and P.B. Mucha: {\em A Lagrangian approach for the incompressible Navier-Stokes equations with variable density.}  Comm. Pure Appl. Math., {\bf 65}, (2012), 1458--1480.

\bibitem{DM5}
R. Danchin and P. B. Mucha:  {\em Incompressible flows with piecewise constant density.} Arch. Ration.
Mech. Anal. 207 (2013), no. 3, 991–1023. doi:10.1007/s00205-012-0586-4


\bibitem{DM1}
R. Danchin and P.B. Mucha: {\em The Incompressible Navier-Stokes Equations in Vacuum},  Comm. Pure Appl. Math. 72 (2019), no. 7, 1351--1385.

\bibitem{DM2}
R. Danchin and P.B. Mucha: {\em Compressible Navier-Stokes equations with ripped density}, 
Comm. Pure Appl. Math., to appear.

\bibitem{DM3}
R. Danchin and P.B. Mucha:
{\em The divergence equation in rough spaces}, J. Math. Anal. Appl. 386 (2012), no. 1, 9–31.

\bibitem{DZ1}
R. Danchin and P. Zhang: {\em Inhomogeneous Navier-Stokes equations in the half-space, with only bounded density.} J. Funct. Anal. 267 (2014), no. 7, 2371–2436.

\bibitem{DZ2} R. Danchin and X. Zhang: {\em On the persistence of H\"older regular patches of density for the inhomogeneous Navier-Stokes equations.} Journal de l'Ecole Polytechnique, 
{\bf 4}, (2017), 781--811.

\bibitem{DHP} 
R.~Denk, M.~Hieber and J.~Pr\"u\ss:
{\em ${\mathcal R}$-boundedness, Fourier multipliers and problems of 
elliptic and parabolic type}. 
\newblock Memoirs of AMS. Vol 166. no. 788,  2003.


\bibitem{Des} B. Desjardins: {\em Global existence results for the incompressible density-dependent Navier-Stokes
equations in the whole space.} Differential Integral Equations 10 (1997), no. 3, 587–598.

\bibitem{ES} 
Y.~Enomoto and  Y.~Shibata:
\newblock {\em On the ${\mathcal R}$-sectoriality and the initial 
boundary value problem for the viscous compressible fluid flow}.
\newblock Funkcial Ekvac., {56}(3), 441--505, 2013.


\bibitem{FZ}
R. Farwig, C. Qian and  P. Zhang: {\em Incompressible inhomogeneous fluids in bounded domains of $R^3$ with bounded density.} J. Funct. Anal. 278 (2020), no. 5, 108394, 36 pp.  


\bibitem{FG}
F. Fanelli and I. Gallagher: {\em Asymptotics of fast rotating density-dependent incompressible fluids in two space dimensions.} Rev. Mat. Iberoam. 35 (2019), no. 6, 1763–1807.

\bibitem{FK}  H. Fujita and T. Kato: {\em On the Navier-Stokes initial value problem I.} Archive for Rational
Mechanics and Analysis,  16 (1964), 269--315.

\bibitem{GGJ}    F. Gancedo and  E. Garcia-Juarez:  
    {\em Global regularity of 2D density patches for inhomogeneous Navier-Stokes,}
    Arch. Ration. Mech. Anal., {\bf 229}(1),  (2018),  339--360. 

\bibitem{HLL}
C. He, J. Li and B. Lu: {\em Global well-posedness and exponential stability of 3D Navier-Stokes equations with density-dependent viscosity and vacuum in unbounded domains.} Arch. Ration. Mech. Anal. 239 (2021), no. 3, 1809–1835.  

\bibitem{Hoff} D. Hoff: \emph{Uniqueness of weak solutions of the Navier-Stokes equations of multidimensional, compressible flow.} SIAM J. Math. Anal. 37 (2006), no. 6, 1742--1760.

\bibitem{HPZ}
J. Huang, M. Paicu and P. Zhang: {\em Global well-posedness of incompressible inhomogeneous fluid
systems with bounded density or non-Lipschitz velocity.} Arch. Ration. Mech. Anal. 209 (2013), no. 2, 631–682. 

\bibitem{Kaz} A. V. Kazhikhov: {\em Solvability of the initial-boundary value problem for the equations of the motion
of an inhomogeneous viscous incompressible fluid.} Dokl. Akad. Nauk SSSR 216 (1974), 1008–1010.

\bibitem{LS} O. Ladyzhenskaya and  A. Solonnikov: {\em Unique solvability of an initial and boundary value problem for viscous incompressible inhomogeneous fluids.} J. Sov. Math. 9 (1978), no. 5, 697–749.
doi:10.1007/BF01085325

\bibitem{Li} J. Li: {\em Local existence and uniqueness of strong solutions to the Navier-Stokes equations with nonnegative density.} J. Differential Equations 263 (2017), no. 10, 6512–6536.
doi:10.1016/j.jde.2017.07.021

\bibitem{LL}
X. Liao and  Y. Liu: {\em Global regularity of three-dimensional density patches for inhomogeneous incompressible viscous flow.} Sci. China Math. 62 (2019), no. 9, 1749–1764.

\bibitem{LZ1} X. Liao and P. Zhang: {\em On the global regularity of the two-dimensional density patch for inhomogeneous incompressible viscous flow.} Arch. Ration. Mech. Anal. 220 (2016), no. 3, 937–981.
doi:10.1007/s00205-015-0945-z

\bibitem{LZ2}  X. Liao and P. Zhang:
{\em Global regularities of two-dimensional density patch for inhomogeneous incompressible viscous flow with general density.}  Comm. Pure Appl. Math., {\bf 72}(4),  (2019),  835--884. 


\bibitem{PLL}  P.-L. Lions: {\em Mathematical topics in fluid
mechanics. Incompressible models}, Oxford Lecture Series in Mathematics and its Applications, {\bf 3}, 1996.

\bibitem{MPesz} P.B. Mucha and  J. Peszek: {\em The Cucker-Smale equation: singular communication weight, measure-valued solutions and weak-atomic uniqueness. } Arch. Ration. Mech. Anal. 227 (2018), no. 1, 273--308.


\bibitem{MP} P.B. Mucha and  T. Piasecki: {\em Stationary compressible Navier-Stokes equations with inflow condition in domains with piecewise analytical boundaries.} Pure Appl. Anal. 2 (2020), no. 1, 123--155

\bibitem{PZ2}
M. Paicu and  P. Zhang: {\em Striated regularity of 2-D inhomogeneous incompressible Navier-Stokes system with variable viscosity.} Comm. Math. Phys. 376 (2020), no. 1, 385–439.

\bibitem{PZZ} M. Paicu, P. Zhang and  Z. Zhang: {\em Global unique solvability of inhomogeneous Navier-Stokes
equations with bounded density.} Comm. Partial Diff. Equations 38 (2013), no. 7, 1208–
1234. doi:10.1080/03605302.2013.780079

\bibitem{PSZ}
T. Piasecki, Y. Shibata and  E. Zatorska: {\em On the maximal Lp-Lq regularity of solutions to a general linear parabolic system.} J. Differential Equations 268 (2020), no. 7, 3332--3369

\bibitem{PR}  B. Piccoli and F. Rossi: {\em Transport equation with nonlocal velocity in Wasserstein spaces: convergence of numerical schemes.} Acta Appl. Math. {\bf 124} (2013), 73–105.

\bibitem{ShiShi} 
Y.~Shibata and  S. Shimizu:
{\it On the $L_p$-$L_q$ maximal regularity of the Neumann problem for  the Stokes equations in a bounded domain.} 
J. Reine Angew. Mat., {\bf 615}, 157--209, 2008.

\bibitem{Simon} J. Simon: {\em Nonhomogeneous viscous incompressible fluids: existence of velocity, density, and pressure.} SIAM J. Math. Anal. 21 (1990), no. 5, 1093–1117. doi:10.1137/0521061

\bibitem{Maja} M. Szlenk: {\em Weak solutions for the Stokes system for compressible fluids with general pressure}, arXiv:2103.08726.   

\bibitem{Zhang20}
P. Zhang: {\em Global Fujita-Kato solution of 3-D inhomogeneous incompressible Navier-Stokes system.} Adv. Math. 363 (2020), 107007, 43 pp.

\end{thebibliography}
\end{document}